\documentclass[12pt]{amsart}
\usepackage[margin=1in]{geometry}                
\usepackage{graphicx}
\usepackage{amssymb}
\usepackage{amsthm}
\usepackage{mathrsfs}
\usepackage{caption}
\usepackage{xcolor}
\usepackage{bbm}
\usepackage{arydshln}

\newtheorem{theorem}{Theorem}[section]
\newtheorem{proposition}[theorem]{Proposition}
\newtheorem{lemma}[theorem]{Lemma}
\newtheorem{definition}[theorem]{Definition}
\newtheorem{corollary}[theorem]{Corollary}

\newtheorem{remark}{Remark}[section]

\numberwithin{equation}{section}

\newcommand{\FF}{\mathbb{F}}

\newcommand{\ZZ}{\mathbb{Z}}

\newcommand{\NN}{\mathbb{N}}
\newcommand{\PP}{\mathbb{P}}

\newcommand{\bbB}{\mathbb{B}}

\newcommand{\cala}{\mathcal{A}}

\newcommand{\calp}{\mathcal{P}}

\newcommand{\calh}{\mathcal{H}}

\newcommand{\one}{\mathbf{1}}

\newcommand{\Span}{\hbox{\rm span}}

\title{Generalized polynomials and hyperplane functions in $(\mathbb{Z}/p^k\mathbb{Z})^n$} 
\author{Izabella {\L}aba and Charlotte Trainor}	
	
\date{\today}

\begin{document}

\begin{abstract}
For $p$ prime, let $\mathcal{H}^n$ be the linear span of characteristic functions of hyperplanes in $(\mathbb{Z}/p^k\mathbb{Z})^n$. We establish new upper bounds on the dimension of $\mathcal{H}^n$ over $\mathbb{Z}/p\mathbb{Z}$, or equivalently, on the rank of point-hyperplane incidence matrices in $(\mathbb{Z}/p^k\mathbb{Z})^n$ over $\mathbb{Z}/p\mathbb{Z}$. Our proof is based on a variant of the polynomial method using binomial coefficients in $\mathbb{Z}/p^k\mathbb{Z}$ as generalized polynomials. We also establish additional necessary conditions for a function on $(\mathbb{Z}/p^k\mathbb{Z})^n$ to be an element of $\mathcal{H}^n$.

\end{abstract}

\maketitle

\tableofcontents


\section{Introduction}


Let $p$ be a prime number, and let $k\in\NN$. We define $R:=\ZZ/p^k\ZZ$, the ring of integers modulo $p^k$, and use $R^\times$ to denote the multiplicative group of invertible elements of $R$.
For $x\in R^n$, we write $x=(x_1,\dots,x_n)$ in terms of coordinates. We also define the inner product on $R^n$ as the $R$-valued function $\langle x,y \rangle=x_1y_1+\dots +x_ny_n$.  

Recall that the projective space $\PP (\ZZ/p\ZZ)^{n-1}$ is defined as the quotient space $(\ZZ/p\ZZ)^n/\sim$, where $\sim$ is the equivalence relation
$$
b\sim b' \ \ \Leftrightarrow \ \ b=\lambda b' \hbox{ for some }\lambda\in (\ZZ/p\ZZ)\setminus\{0\}.
$$
When $k>1$, the projective space over $R^n$ must be defined a little bit more carefully. Define the $(n-1)$-dimensional sphere $\mathbb{S}^{n-1}(R)$ to be the set of all elements of $R$ that have at least one invertible component. In particular, $\mathbb{S}^{0}(R)=R^\times$. We then define
$$
\PP R ^{n-1}= \mathbb{S}^{n-1}(R)/\mathbb{S}^{0}(R).
$$
We will refer to the elements of $\PP R ^{n-1}$
as {\em nondegenerate directions} in $R^n$. Thus, two elements $b,b'$ of $\mathbb{S}^{n-1}(R)$ define the same direction if and only if 
\begin{equation}\label{equidirection}
b=\lambda b'\hbox{ for some }\lambda\in R^\times.
\end{equation}
This is how directions in $R^n$ are often defined in the literature, see e.g. \cite{HW}. All directions will be assumed to be nondegenerate unless explicitly stated otherwise.

A {\em hyperplane} is a set of the form 
$$H_b(a)=\{a\in R^n:\ \langle x-a,b\rangle =0 \},$$ 
for some $a\in R^n$ and a nondegenerate direction $b\in \PP R ^{n-1}$. (Note that the equality 
$\langle x-a,b\rangle =0$ should hold in $R$ and not just modulo $p$.) When $a=0$, we write
$H_b=H_b(0)$. We will sometimes refer to $H_b$ as {\em homogeneous hyperplanes}, and to $H_b(a)$ as {\em affine hyperplanes}. We also define
$$
\mathcal{H}^n=\hbox{\rm span}_{\ZZ/p\ZZ}\{{\bf 1}_{H_b(a)}:\ a\in R^n,b\in \PP R ^{n-1}\},
$$
considered as a set of functions from $R^n$ to $\ZZ/p\ZZ$.

\begin{definition}\label{def-hyperplane}
Let $R=\ZZ/p^k\ZZ$, where $p$ is a prime and $k\in\NN$. 

\begin{itemize}
\item[(i)]
The \emph{point-hyperplane incidence matrix} of \(R^n\) is the matrix 
\(W_{p^k,n}\), with rows and columns indexed by \(x\in R^n\), such that 
\begin{displaymath} (W_{p^k,n})_{x,y}=\left\{\begin{array}{ll}
1 & \text{if }\langle x,y\rangle=0, \\
0&\text{otherwise.}
\end{array}\right.
\end{displaymath}

\item[(ii)]
The \emph{reduced point-affine hyperplane incidence matrix} of \(R^n\) is the matrix 
\(\mathcal{A}^*_{p^k,n}\), with rows indexed by \((x,a)\in R^n\times R^n\) and columns indexed by $b\in \PP R^{n-1}$, such that
\begin{displaymath} (\cala^*_{p^k,n})_{(x,a),b}=\left\{\begin{array}{ll}
1 & \text{if }x\in H_b(a), \\
0&\text{otherwise.}
\end{array}\right.
\end{displaymath}

\item[(iii)]
The \emph{reduced point-hyperplane incidence matrix} of \(R^n\) is the matrix $W_{p^k,n}^*$ with rows indexed by $b\in \PP R^{n-1}$ and columns indexed by \(x\in R^n\), such that 
\begin{displaymath} (W_{p^k,n}^*)_{x,b}=\left\{\begin{array}{ll}
1 & \text{if }x\in H_b, \\
0&\text{otherwise.}
\end{array}\right.
\end{displaymath}

\end{itemize}

\end{definition}

Note that the equation $\langle x,y\rangle=0$ in (i) does not define a hyperplane in our sense if $y$ is not a direction; however, we use the terminology above for consistency with the existing literature such as \cite{DD}.

We are interested in upper and lower bounds on the rank of these matrices over $\ZZ/p\ZZ$.
For \(k=1\), the rank of \(W_{p,n}\) is known as a special case of the results in \cite{GD}, \cite{MM}, \cite{Smith}.

\begin{theorem}[\cite{GD}, \cite{MM}, \cite{Smith}]\label{k-is-1}
\label{WpnRank} For \(p\) prime and \(n\in\NN\), 
\[ \emph{rank}_{\ZZ/p\ZZ}(W_{p,n})={p+n-2\choose n-1}+1. \]
\end{theorem}

Theorem \ref{k-is-1} can be deduced from a characterization of hyperplane functions in $\mathbb{F}_p^n$ in terms of polynomials. Specifically, when $k=1$, $\mathcal{H}^n$ is identical to $\mathbb{F}_p[x_1,\dots,x_n]$, the space of all polynomials in  $n$ variables of total degree at most $p-1$ over $\mathbb{F}_p$. Moreover, the subspace $\mathcal{H}^n_0$ spanned by homogeneous hyperplanes is identical to the linear span of all homogeneous polynomials in $\mathbb{F}_p[x_1,\dots,x_n]$ of degree exactly $p-1$, together with the constant function. Counting all such polynomials produces the bound in Theorem \ref{k-is-1}. We provide the full argument in Section \ref{phi-and-hyperplanes}.

For $k\geq 2$, this method is no longer feasible. By Fermat's Little Theorem, a polynomial over $R$ can have degree at most $p-1$ in each variable, hence there are not sufficiently many polynomials to span all hyperplane functions. We remedy this by using binomial coefficients as generalized polynomial functions. This allows us to define generalized polynomials of degree up to $p^k-1$, which is sufficient to span $\mathcal{H}^n$. Binomial coefficients were used in lieu of polynomials in \cite{CS1} for the purpose of extending the Ellenberg-Gijswijt bound on cap sets \cite{EG} to $R^n$; see also \cite{CS2} for an argument based on a more abstract concept of generalized polynomials, and \cite{CS3} for a third approach to cap sets in $R^n$ and a discussion of the relationship between these methods. We are not aware, however, of any previous applications of similar methods to studying hyperplane functions.

In Proposition \ref{hyper-to-phi}, we prove that hyperplane functions in $R^n$ are, in this sense, generalized $n$-variate polynomials of degree up to $p^k-1$. This implies our first theorem.

\begin{theorem}\label{main-thm-affine}
For $p$ prime and $k,n\in\NN$, we have
$$
\emph{dim}_{\ZZ/p\ZZ}(\mathcal{H}^n)
=\emph{rank}_{\ZZ/p\ZZ}(\mathcal{A}^*_{p^k,n})\leq {p^k-1+n\choose n}.
$$

\end{theorem}

However, unlike for $k=1$, hyperplane functions in $R^n$ with $k\geq 2$ need not span all such generalized polynomials.
In fact, we have the following bound, which is strictly lower than that in Theorem \ref{main-thm-affine}
when $k\geq 2$ and $n$ is small relative to $p^k$.

\begin{theorem}\label{main-upper-bound-1}
Let $p$ be prime, and let $k,n\in\NN$. Then
\begin{equation} 
\label{upperbound-1}
\emph{rank}_{\ZZ/p\ZZ}(\mathcal{A}_{p^k,n}^*)\leq (2n){\lfloor p^k/2\rfloor+(n-1)(p-1)+n\choose n},
\end{equation}
\end{theorem}

Theorems \ref{main-thm-affine} and \ref{main-upper-bound-1} imply upper bounds on the ranks of 
$W^*_{p^k,n}$ and $W_{p^k,n}$, via the next proposition.

\begin{proposition}
\label{compare-both}
Let $n\in\NN$, $n\geq 2$. Then
\[ \emph{rank}_{\ZZ/p\ZZ}(W_{p^k,n})\leq1+k\cdot\emph{rank}_{\ZZ/p\ZZ}(W_{p^k,n}^*), \]
$$
\emph{rank}_{\ZZ/p\ZZ}(W_{p^k,{n+1}}^*)
\leq 2 (k+1) \cdot\emph{rank}_{\ZZ/p\ZZ}(\cala_{p^k,n}^*) .
$$

\end{proposition}

Theorem \ref{main-upper-bound-1} raises the question of how we can tell whether a given generalized polynomial of degree at most $p^k-1$ is a hyperplane function. Our generalized polynomials share many geometric properties of hyperplane functions. For example, if $L,L'$ are two parallel lines in $R^n$, then $|L\cap H|\equiv |L'\cap H|$ mod $p$ for any hyperplane $H$; we prove in Proposition \ref{parallel-lines} that an appropriate analogue of this holds for generalized polynomials of degree up to $p^k-1$. Nonetheless, we are able to find a class of functions on $R^n$ we call \emph{fans} that are orthogonal over $\ZZ/p\ZZ$ to all hyperplane functions, but not to some of our generalized polynomials of degree up to $p^k-1$. Essentially, this test identifies generalized polynomials that behave like hyperplane functions on each scale separately, but the directions are not consistent between the scales. Since the statement of the result requires some notation, we postpone it to Section \ref{sec-fan}. While a generalized polynomial must satisfy our orthogonality condition in order to be a hyperplane function, we do not know whether this condition is also sufficient.

Our interest in hyperplane functions is motivated in part by the recent work of Dhar and Dvir \cite{DD}, where a connection was established a connection between point-hyperplane incidence matrices and the Kakeya problem. For $k=1$, Dhar and Dvir used Theorem \ref{k-is-1} 
to give a new proof of Dvir's result \cite{Dvir}
 that a Kakeya set \(S\subset (\ZZ/p\ZZ)^n\) must satisfy \(|S|\gtrsim_\epsilon p^{n-\epsilon}\)
for any \(\epsilon>0\). 
They were then able to extend this matrix-based argument to prove the Kakeya conjecture in $\ZZ/N\ZZ$ for squarefree $N$.
In $R=\ZZ/p^k\ZZ$ with $k\geq 2$, Dhar and Dvir were still able to bound the size of Kakeya sets in $R^n $ from below by the $\mathbb{F}_p$-rank of $W_{p^{k},n}^*$. (In \cite[Theorem 1.6]{DD}, the authors refer to the rank of $W_{p^{k},n}$; however, their argument uses the matrix $W_{p^{k},n}^*$ instead. The two ranks are not equal, but they are comparable; see Lemma \ref{ranksUnequalFor2} and Proposition \ref{W-times-2}.)

Unfortunately, relatively little has been known about the $\mathbb{F}_p$-rank of point-hyperplane incidence matrices in $R^n$. 
Dhar and Dvir \cite[Lemma 5.3]{DD} observe that the rank of $W_{p^k.n}$ is bounded from below by the size of a maximal {\em matching vector family} in $R^n$. Combining this with the results of \cite{DGY11,YGK12} yields a lower bound on the rank of $W_{p^k,n}$ of the order $p^{kn/2}$, therefore a lower bound of the same order on the size of Kakeya sets in $R^n$.  Dhar and Dvir observe further that, in light of an upper bound on the size of matching vector families given in \cite{DH13b}, this method cannot yield significantly better lower bounds.

The Kakeya conjecture in $R^n$ was eventually resolved by Arsovski \cite{Ar}, based on a comparison of the size of Kakeya sets to the rank of a different matrix that, in general, may have higher rank than $W_{p^k,n}$. Subsequently, Dhar \cite{Dhar} proved the Kakeya conjecture in $\ZZ/N\ZZ$ for general $N$, with further progress in \cite{Dhar2, Dhar3}.

The question of the rank of the point-hyperplane incidence matrices in Definition \ref{def-hyperplane} was left open. While this is no longer needed for the Kakeya problem, we believe it to be of independent interest, as it provides a good testing ground for variants of the polynomial method that rely on generalized polynomials.

This paper is organized as follows.
We study the relationships between the ranks of the different incidence matrices in Section
\ref{sec:prelim}. Proposition \ref{compare-both} follows from Propositions \ref{W-times-2} and \ref{compareAW}.
In Section \ref{section-define-phi}, we define our generalized polynomials in one variable based on binomial coefficients.
The rest of Section \ref{section-define-phi}, as well as Section \ref{discrete-derivatives},
are dedicated to the study of the properties of these functions. An important feature of a ``generalized polynomial" of degree $m$ is that its derivatives of order $m+1$ should vanish; we prove in Lemma \ref{lemma:Dm0} that our binomial functions have this property.

In Section \ref{sec-multidim-phi}, we extend our generalized polynomials to $R^n$ and prove that they are, again, well behaved with respect to discrete derivatives. We also prove that hyperplane functions in $R^n$ are generalized polynomials of degree at most $p^k-1$. In particular, Theorem \ref{main-thm-affine} follows from Proposition \ref{hyper-to-phi}. We note that, while an {\it ad hoc} application of binomial coefficients was sufficient in \cite{CS1}, we need to develop our theory more systematically.

A major difficulty in working with binomial coefficients is that they do not have good multiplicative properties. This is one reason why there is no straightforward way to adapt the methods from the $k=1$ case to our setting (and why, for the time being, we are only able to prove partial results). This turns out to be more than just a technical issue. Our results in Section \ref{sec-degree-lowering} show that the behaviour of our generalized polynomials is genuinely different than that of classical polynomials. For example, $(xy)^m=x^my^m$ is a bivariate polynomial of degree $2m$; on the other hand, if $f$ is a generalized polynomial of degree $m$ on $R$, then the degree of $f(xy)$ cannot be much larger than $m$. This degree reduction is the main idea behind the proof of Theorem \ref{main-upper-bound-1} in Section \ref{sec-main-upper-bound-1}.

Finally, in Section \ref{sec-fan} we study the geometric properties of lines and hyperplanes in $R^n$, and develop a test that (at least in some cases) allows us to determine that a given generalized polynomial is not a hyperplane function.

Throughout this article, we will observe the following conventions. 
Arithmetic operations and equalities for elements of $R$ will be defined in $R$, that is, modulo $p^k$. For example, if $a,b\in R$, the equality $a=b$ will mean that $a\equiv b$ mod $p^k$. When we work with functions with values in $\ZZ/p\ZZ$ (such as the $\phi_m$ functions defined in (\ref{def:phi})), all arithmetic operations and equalities involving such functions will be understood to hold in $\ZZ/p\ZZ$. In expressions such as $a f(x)$, where $a,x\in R$ and $f$ is a function $R\rightarrow \ZZ/p\ZZ$, we will interpret $a$ as the function $a\to (a$ mod $p)$, so that $af(x)$ refers to the  function $(a$ mod $p)f(x)$ with values in $\ZZ/p\ZZ$. The inner product in $R$ is an $R$-valued function, so that $\langle x,y \rangle=c$ means that $x_1y_1+\dots +x_ny_n\equiv c$ mod $p^k$ and not just mod $p$.
On the other hand, if $f,g$ are two functions from $R^n$ to $\ZZ/p\ZZ$, their inner product
$$
\langle f,g\rangle = \sum_{x\in R^n} f(x) g(x)
$$
takes values in $\ZZ/p\ZZ$.

In line with our use of functions with range in $\ZZ/p\ZZ$, whenever we refer to the rank of a matrix, the span of a set of vectors, or the dimension of a linear space of functions, this rank, span, or dimension is taken over $\ZZ/p\ZZ$ unless explicitly stated otherwise.

For $m\in\NN$, we write $[m]=\{0,1,\dots,m-1\}\subset\ZZ$. We will distinguish between $R$, a ring with addition and multiplication mod $p^k$, and $[p^k]$, a set of integers where addition and multiplication are inherited from $\ZZ$ (so that $[p^k]$ is not closed under these operations). 
Exponents, indices, etc.~will always be integers unless stated explicitly otherwise. For example, if $\ell$ is the degree of a polynomial or a generalized polynomial, we will write $\ell\in [p^k]$ and not $\ell\in R$. 

We use the notation $|S|$ to denote the cardinality of a set $S$, and the notation $p^j\parallel a$ to mean $p^j\mid a$ but $p^{j+1}\nmid a$. We also use subscripts $1,\dots,n$ to denote both the coordinates $x=(x_1,\dots,x_n)$ of a point $x\in R^n$ and the $p$-adic digits in the expansion $x=\sum_{j=0}^{k-1}x_jp^j$ of an element $x\in R$. This should not cause confusion, since we will only use one of the above at a time and the meaning will be clear from context. Whenever we mention the $p$-adic expansion or $p$-adic digit of a number $x$, we refer to the unique expansion $x=\sum_{j=0}^k x_jp^j$ with $x_j\in\{0,1,\dots,p-1\}$ for all $j$.


\section{Relationships between incidence matrices}
\label{sec:prelim}


We first observe that
\begin{equation}
\label{compareWs}
 \mathrm{rank}(W_{p^k,n}^*)\leq \mathrm{rank}(W_{p^k,n}),
\end{equation}
 since the rows of $W_{p^k,n}^*$ form a subset of the rows of $W_{p^k,n}$. Lemma \ref{ranksUnequalFor2} shows that the inequality can be strict for $k\geq 2$.

\begin{lemma}\label{ranksUnequalFor2}
If \(k\in\NN\) and \(k\geq 2\), then
$\emph{rank}(W_{p^k,2})>\emph{rank}(W_{p^k,2}^*) $.
\end{lemma}

\begin{proof} 
All directions in \(R^2\) can be represented by one of the elements of the set
\[ \mathcal{D}=\{(1,i): i\in R\} \cup \{(jp,1): j\in\{0,1,\dots,p^{k-1}-1\}\}. \]
Given a direction $b\in R^2$, define
$$
L_b =\{ t b:\ t\in R\}.
$$

Given \(b\in\mathcal{D}\), there is some \(c\in \mathcal{D}\) such that \(H_b=\text{span}(c):=\{\lambda c:\lambda\in R\}\). Let 
\[ \mathcal{H}=\{\one_{H_b}:b\in\mathcal{D}\}=\{\one_{L_b}:b\in\mathcal{D}\}. \]
Then \(\mathcal{H}\) consists of exactly the rows of \(W_{p^k,2}^*\), and is a subset of the rows of \(W_{p^k,2}\).

Let $y=(p^{k-1},0)$, then the indicator function of 
$H_y:=\{x\in R^n:\langle x,y\rangle=0 \}$ is a row of $W_{p^k,2}$. We claim that
\[ \one_{H_y}\not\in \text{span}\mathcal{H}.\]
Assume towards contradiction that there are scalars \(\alpha_i\), \(\beta_j\) such that
\begin{equation}\label{n2-e1}
 \one_{H_y}(x)=\sum_{i=0}^{p^k-1}\alpha_i\one_{L_{(1,i)}}(x)+\sum_{j=0}^{p^{k-1}-1}\beta_j\one_{L_{(pj,1)}}(x).
 \end{equation}
We first evaluate (\ref{n2-e1}) at $x=(pj,1)$ for  \(j\in\{0,\dots,p^{k-1}-1\}\).
Since $(pj,1)\in H_y$ but
\[
(pj,1)\not\in L_{(1,i)},\ \ (pj,1)\not\in L_{(p\ell,1)} \text{  if } j\neq \ell, \]
 it follows that \(\beta_j=1\) for all \(j\). 
Now evaluate (\ref{n2-e1}) at $x=(0,p^{k-1})$. Since 
\[ (0,p^{k-1})\not\in L_{(1,i)} \text{ for all \(i\), but } (0,p^{k-1})\in L_{(pj,1) }\text{  for all } j, \] we have
\[ \sum_{i=0}^{p^k-1}\alpha_i\one_{L_{(1,i)}}(0,p^{k-1})+\sum_{j=0}^{p^{k-1}-1}\beta_j\one_{L_{(pj,1)}}(0,p^{k-1}) =p^{k-1}= 0\bmod{p}.\]
This is a contradiction, as \((0,p^{k-1})\in H_y\). 
\end{proof}

In the next proposition, we provide a partial converse to the inequality in~(\ref{compareWs}). 

\begin{proposition}\label{W-times-2}
Let $n\geq2$ and $k\geq1$. Then 
\[ \emph{rank}(W_{p^k,n})\leq1+\sum_{j=1}^k\emph{rank}(W_{p^j,n}^*),  \]
and consequently, 
\[ \emph{rank}(W_{p^k,n})\leq1+k\cdot\emph{rank}(W_{p^k,n}^*).  \]
\end{proposition}

\begin{proof}
Recall that the columns of $W_{p^k,n}$ are indexed by $b\in R^n$. Partition these columns by the sets 
\[ B_j = \{ b'\in R^n: b'=p^jb,\; b\neq0\bmod{p}\},\]
and let $W^{(j)}$ be the submatrix of $W_{p^k,n}$ consisting of columns indexed by $b'\in B_j$. Then 
\[ \text{rank}(W_{p^k,n})\leq\sum_{j=0}^k\text{rank}(W^{(j)}).\]
Note that the only vector in $B_k$ is the zero vector, and so $W^{(0)}$ is a just a column of all $1s$, which has rank $1$. Thus to prove the proposition, it suffices to show that for $j\in[k]$, we have $\text{rank}(W^{(j)})\leq\text{rank}(W_{p^{k-j},n})$. We show that this actually holds with equality. 

Let $j\in [k]$. The column of $W^{(j)}$ corresponding to $b'\in B_j$ is the indicator vector of $\{x\in R^n:\langle x,b'\rangle=0\bmod{p^k}\}$. Recalling that $b'=p^jb$ for a direction $b$, we have 
\begin{equation}
\label{hypequiv}
 \langle x,b'\rangle=0\bmod{p^k} \text{ \; if and only if \; } \langle x,b\rangle=0\bmod{p^{k-j}}.   
\end{equation}  
Notice that the latter equation only depends on $x\bmod{p^{k-j}}$; we will use this observation to partition the rows of $W^{(j)}$. 

For $\ell\in[k]$, let
$\overline{R}_{\ell}^n$ be the set of $x\in R^n$ so that for each $i\geq \ell$, the $i$-th $p$-adic digit of each component of $x$ is zero. Consider the sets
\[ X_u:= up^{k-j}+\overline{R}_{k-j}^n, \quad u\in\overline{R}_j^n. \]
Let $W^{(j)}_u$ be the submatrix of $W^{(j)}$ consisting of rows indexed by $x\in X_u$. By definition, for each $u$, the set $\{x\mod{p^{k-j}}:x\in X_u\}$ can be identified with $R^n_{k-j}$.
Similarly, the set $\{b:p^jb\in B_j\}$ can be identified with the set of directions of $R_{k-j}^n$. Combining these observations with the equivalence in~(\ref{hypequiv}), we see that $W_u^{(j)}$ is the same matrix as $W^*_{p^{k-j},n}$. As this is true for each $u$, the matrix $W^{(j)}$ is formed by vertically concatenating copies of $W^*_{p^{k-j},n}$. Thus it has the same rank as $W^*_{p^{k-j},n}$, as claimed.
\end{proof}

\begin{proposition}
\label{compareAW}
Let $n\in\NN$, $n\geq 2$. Then
\begin{equation}\label{compareAW-eq}
\emph{rank}(\cala_{p^k,n}^*)\leq\emph{rank}(W_{p^k,{n+1}}^*)
\leq 2 (k+1) \cdot\emph{rank}(\cala_{p^k,n}^*) .
\end{equation}

\end{proposition}

\begin{proof}
We write directions \(b\in R^{n+1}\) as  \(b=(\widetilde{b},b_{n+1})\), with \(\widetilde{b}\in R^{n}\). 
By a mild abuse of notation, we identify \(b\) with an element of \(\PP R^{n}\).  We use a similar convention for points $x\in R^{n+1}$.

We first prove that $\text{rank}(\cala_{p^k,n}^*)\leq\text{rank}(W_{p^k,{n+1}}^*)$. 
Any affine hyperplane in $R^n$ can be written as
\begin{equation}
\label{Htilde1}
\widetilde{H}_b=\left\{\widetilde{x}\in R^{n}:\langle \widetilde{x},\widetilde{b}\rangle =-b_{n+1}\right\},
\end{equation}
where $\widetilde{b}\in\PP  R^{n-1}$ is a direction, and $b_{n+1}\in R$. For any such $(\widetilde{b},b_{n+1})$, let
\begin{equation} 
\label{Htilde}
H_b=\left\{x=(\widetilde{x},x_{n+1})\in R^{n+1}:\langle \widetilde{x},\widetilde{b}\rangle+b_{n+1}x_{n+1}=0\right\}, 
\end{equation}
so that 
$\widetilde{H}_b\times \{1\}=H_b\cap\{x\in R^{n+1}:x_{n+1}=1\}$. 
Consider the submatrix of \(W_{p^k,n+1}^*\) obtained by restricting to rows indexed 
by \((\widetilde{b},b_{n+1})\in \mathcal{B}:=\PP R^{n-1}\times R\) and columns indexed by $x\in R^n\times\{1\}$. By the above correspondence, this submatrix is a copy of \(\mathcal{A}_{p^k,n}^*\), giving the desired bound.  

When considering the converse of this argument, it might be possible for a set of columns of the submatrix defined above to be linearly dependent even if the corresponding columns of the larger matrix \(W_{p^k,n}^*\) are linearly independent. We remedy this by considering linear independence on each scale separately.

For $j\in\{0,1,\dots,k\}$, let $X_j = \{x\in R^{n+1}:  x_{n+1}=p^jy, \; y\neq0\bmod{p}\}$.
Let $W^{(j)}$ be the submatrix of $W_{p^k,n+1}^*$ formed by restricting to the columns with $x\in X_j$. Then 
\[ \text{rank}(W^*_{p^k,n+1})\leq\sum_{j=0}^{k} \text{rank}(W^{(j)}).\]
We will show that $\text{rank}(W^{(j)})\leq 2\cdot\text{rank}(A_{p^k,n}^*)$ for each $j\in\{0,1,\dots,k\}$, implying
 the second bound in (\ref{compareAW-eq}).

Let $W^{(j)}_1$ be the submatrix formed by restricting to the rows indexed by $b\in \mathcal{B}$, and let $W^{(j)}_2$ be the submatrix consisting of the remaining rows. 
Clearly, $\text{rank}(W^{(j)})\leq \text{rank}(W^{(j)}_1)+\text{rank}(W^{(j)}_2)$.
It therefore suffices to prove that
\begin{equation}
\label{rankWj}
\text{rank}(W^{(j)}_i)\leq \text{rank}(\mathcal{A}_{p^k,n}^*)\hbox{ for }i=1,2.
\end{equation}
We first prove (\ref{rankWj}) for $i=1$. 
For each $b=(\widetilde{b},b^{n+1})  \in \mathcal{B}$, let 
$H_b=\{x\in R^{n+1}:\langle x,b\rangle=0\}$, and let
\[ 
\widetilde{H}_{b,j}=\{\widetilde{x}\in R^n: \langle \widetilde{x},\widetilde{b}\rangle=-p^jb_{n+1}\},
\]
so that $\widetilde{H}_{b,j}\times \{p^j\}=H_b\cap\{x\in R^{n+1}:x_{n+1}=p^j\}$.
We first note that
\begin{equation}\label{rankWj-e1}
\text{rank}(W_1^{(j)})\leq \text{dim}\left(\text{span}\{\one_{H_b\cap X_j}:b\in \mathcal{B}_j\}\right),
\end{equation}
where $\mathcal{B}_j=\{b\in \mathcal{B}:b_{n+1}\in[p^{k-j}]\}$. This is because, for $x\in X_j$, the value of $\one_{H_b}(x)$ is determined uniquely by $\widetilde{b}$ and the first $k-j$ digits in the $p$-adic expansion of $b_{n+1}$. Next, we prove that
\begin{equation}\label{rankWj-e2}
\text{dim}\left(\text{span}\{\one_{H_b\cap X_j}:b\in \mathcal{B}_j\}\right)
\leq 
\text{dim}\left(\text{span}\{\one_{\widetilde{H}_{b,j}}:b\in \mathcal{B}_j\}\right).
\end{equation}

For $j=k$, we have $X_k=\{(\tilde{x},0):\tilde{x}\in R^{n}\}$ and
$\mathcal{B}_k=\{(\tilde{b},0):\tilde{b}\in R^{n}\}$, so that for $b\in \mathcal{B}_k$ we have
$H_b\cap X_k=\widetilde{H}_{b,k}\times\{0\}$ and the claim is clear. 

We now assume that $j\leq k-1$.
Suppose that there are scalars $c_b$ so that 
\begin{equation}
\label{lincombHtilde}
\sum_{b\in \mathcal{B}_j}c_b\one_{\widetilde{H}_{b,j}}=0.
\end{equation}
We will show that $\sum_{b\in \mathcal{B}_j}c_b\one_{H_b\cap X_j}=0$
as well.  For $s\in [p^{k-j}]$, $s\neq0\bmod{p}$, define 
\[ X_{j,s}=\{x\in X_j: x_{n+1}=sp^j\}. \] 
First, we note that as $s$ is invertible, 
\begin{equation} 
\label{HbXjs}
H_b\cap X_{j,s}=\{(s\widetilde{x},sp^j): \widetilde{x}\in \widetilde{H}_{b,j}\} 
\end{equation}
and as the $X_{j,s}$ form a partition for $X_j$, 
\[ \one_{H_b\cap X_j}=\sum_s\one_{H_b\cap X_{j,s}}.\]
Then 
\begin{align*}
\sum_{b\in\mathcal{B}_j}c_b\one_{H_b\cap X_j}=\sum_{b\in\mathcal{B}_j}c_b\sum_{s}\one_{H_b\cap X_{j,s}}=\sum_{s}\left(\sum_{b\in\mathcal{B}_j}c_b\one_{H_b\cap X_{j,s}}\right)
\end{align*}
But each term in the outermost sum of the right-hand side of this equation is equal to zero, by~(\ref{HbXjs}) and~(\ref{lincombHtilde}). Thus $\sum_{b\in\mathcal{B}_j}c_b\one_{H_b\cap X_j}=0$, proving (\ref{rankWj-e2}).

Combining (\ref{rankWj-e1}) and (\ref{rankWj-e2}), we get
\[ \text{rank}(W_1^{(j)})\leq \text{dim}\left(\text{span}\{\one_{\widetilde{H}_{b,j}}(b):b\in B_j\}\right)\leq \text{rank}(\mathcal{A}_{p^k,n}^*). \]
as claimed.

To prove (\ref{rankWj}) for $i=2$, we observe that if $b=(\widetilde{b},b_{n+1})\not\in \mathcal{B}$, then $\widetilde{b}$ is not a direction in $R^n$, hence none of $b_1,\dots,b_n$ are invertible. Since $b$ is a direction in $R^{n+1}$, we have $b_{n+1}\in R^\times$, so that $b=(b_1,\widetilde{b})$ for a direction $\widetilde{b}\in R^n$. The desired bound follows by the same argument as above with the first and last coordinates interchanged.
\end{proof}


\section{The binomial phi functions}
\label{section-define-phi}

\subsection{Definitions}
In this section, we work in $R=\ZZ/p^k\ZZ$ and use the representatives \(R=\{0,1,2,\dots,p^k-1\}\).
Given two elements $x,y\in R$, we will write that $x< y$, $x\leq y$, etc.~if the stated inequality holds for the representatives of $x,y$ chosen above. 

\begin{definition}\label{def-phi-functions}
For $m\in[p^k]$, we define the functions $\phi_m:R\to \ZZ/p\ZZ$ by
\begin{equation}\label{def:phi}
 \phi_m(x)={x\choose m}\bmod{p},
 \end{equation}
with the convention that \({0\choose 0}=1\) and \({a\choose b}=0\) for \(a<b\).
We also define for all $x\in R$,
\begin{equation}\label{def:phi-2}
\phi_m(x)= 0 \hbox{ if }m<0\hbox{ or }m\geq p^k.
\end{equation}
\end{definition}

The binomial coefficients above are well defined by Lucas's Theorem, which we recall here for the reader's convenience.

\begin{theorem}[{\bf Lucas's Theorem}]\label{lucas}
Let $p$ be prime. Let $m,n$ be nonnegative integers with $p$-adic expansions $m=\sum_{j=0}^\ell m_j$ and $n=\sum_{j=0}^\ell n_jp^j$, where $m_j,n_j\in [p]$. Then, with the same convention as above, 
\[ {m\choose n}\equiv \prod_{j=0}^\ell{m_j\choose n_j}\bmod{p} .\]
\end{theorem}

\begin{proposition}\label{binomial-phi} 
If \(x,y\in\ZZ_{\geq 0}\) satisfy \(x\equiv y\bmod{p^k}\), then \({x\choose m}\equiv {y\choose m}~\bmod{p}\).
Consequently, $\phi_m$ are well-defined as functions on $R$.
They satisfy the recurrence relations
\begin{equation}
\label{eq:phirec}
\begin{split}
&\phi_0(x)=1\hbox{ for all }x\in R, \quad \phi_m(0)=0\hbox{ for all }m\neq 0, \\
&  \phi_m(x+1)= \phi_m(x)+\phi_{m-1}(x)
\hbox{ for } m\in[p^k].
\end{split}
\end{equation}
Furthermore, if $m\in[p^k]$ and $x\in R$ have the $p$-adic expansions $m=\sum m_ip^i$ and $x=\sum x_ip^i$, then 
\begin{equation}\label{phi-digits}
\phi_m(x)=\prod_{i=0}^{k-1} \phi_{m_i}(x_i).
\end{equation} 
\end{proposition}

\begin{proof}
The first conclusion is trivial when $m<0$ or $m\geq p^k$, since then $\phi_m(x)=0$ for all $x$. Assume now that $m\in[p^k]$ and that $x\equiv y$ mod $p^k$ for some \(x,y\in\ZZ_{\geq 0}\). Then the $p$-adic expansions $x=\sum x_jp^j$ and $y=\sum y_jp^j$ satisfy $x_j=y_j$ for $0\leq j\leq k-1$, and the conclusion follows from Lucas's Theorem.

Part (\ref{eq:phirec}) follows directly from (\ref{def:phi}), (\ref{def:phi-2}), and Pascal's identity for binomial coefficients. Finally, (\ref{phi-digits}) is Lucas's Theorem again.
\end{proof}

%
%

We will view the functions $\phi_m$ as ``generalized polynomials'' on $R$. 
For $m=0,1,\dots,p-1$, we will see that $\phi_m$ is in fact a polynomial of degree $m$ (Corollary \ref{n-dull} with $n=1$). We have $\phi_0(x)=1$ and $\phi_1(x)=x$ for all $x$, but $\phi_m$ with $2\leq m<p$ need not be either homogeneous or monic. For $m\geq p$, (\ref{def:phi}) still makes sense and defines additional functions that can be thought of as ``polynomial'' of degree $m$, for example in the sense of \cite{Leibman}.

Unlike for actual polynomials, there is no canonical choice of homogeneous generalized polynomials on $R^n$. For example, we could have defined $\phi_m(x):={x+m\choose m}$ instead of (\ref{def:phi}), and all our proofs would have been essentially the same with only slightly more complicated calculations. We further note that the recurrence relation  (\ref{eq:phirec}) could be used as an alternative (but equivalent) definition of phi functions.

For $k=1$, the polynomials $1,x,x^2,\dots,x^{p-1}$ are linearly independent functions on $\ZZ/p\ZZ$, therefore form a linear basis for the space of all functions on $\ZZ/p\ZZ$. We now prove that the same is true for the functions $\phi_m$ for general $k$.

\begin{lemma}
\label{lemma:phiLinInd} {\bf (Linear independence of $\phi_m$)}
Let \(\Phi\) be the \(p^k\times p^k\)  matrix with columns indexed by $x\in R$ and rows by $m\in [p^k]$, and with entries
\[ \Phi_{m,x}=\phi_m(x). \]
Then  \(\Phi\) is a nonsingular upper triangular matrix,
with $\Phi_{m,m}={m\choose m}=1$ and $\Phi_{m,x}=0$ for $x<m$. Consequently, 
 the functions $\{\phi_m\}_{m\in[p^k]}$ 
are linearly independent over \(\ZZ/p\ZZ\), and form a basis for the space of all functions from $R$ to $\ZZ/p\ZZ$.

\end{lemma}

\begin{proof}
We clearly have $\Phi_{m,m}={m\choose m}=1$ for all $m\in[p^k]$. If $x,m\in[p^k]$ with $x<m$, then at least one $p$-adic digit of $x$ must be smaller than the corresponding $p$-adic digit of $m$, so that ${x\choose m}=0$ by Lucas's Theorem. It follows that $\Phi$ is an upper triangular matrix, nonsingular since all its diagonal entries are equal to 1.
Since the $m$-th row of $\Phi$ is the list of values of $\phi_m(x)$ as $x\in R$, the linear independence of the rows of $\Phi$ implies the linear independence of $\phi_m$ with $m\in[p^k]$. In particular, the linear span of $\{\phi_m\}_{m\in[p^k]}$ over $\ZZ/p\ZZ$ has dimension $p^k$. Since this is also the dimension of the the space of all functions from $R$ to $\ZZ/p\ZZ$, the last statement follows.
\end{proof}

\subsection{Properties of phi functions}

Vandermonde's Identity (\ref{phi-e1}) is the phi-function analogue of the binomial expansion of $(x+y)^m$. Unfortunately, the simple polynomial formula $(xy)^m=x^my^m$ has a far less transparent analogue (\ref{phi-e2}) for phi functions. A significant amount of work in the sequel will go towards studying the multiplicative properties of $\phi_m$.

\begin{lemma}\label{phiadditive} (Vandermonde's Identity)
For $m\in[p^k]$ and $x,y,b\in R$, we have 
\begin{equation}\label{phi-e1}
\phi_m(x+y)=\sum_{i=0}^m\phi_i(x)\phi_{m-i}(y),
\end{equation}
\begin{equation}\label{phi-e2}
\phi_m(bx)=\sum_{i_1+\dots+i_b=m}\phi_{i_1}(x)\dots \phi_{i_b}(x).
\end{equation}

\end{lemma}

\begin{proof}
Equation (\ref{phi-e1}) is known in the literature, but we include the short proof for completeness.
For $m=0$, the only pair $i,j$ with $i+j=m$ is $i=j=0$, and $\phi_0(x+y)=1=\phi_0(x)\phi_0(y)$. 
For $m=1,\dots,p^k-1$, we prove (\ref{phi-e1}) by induction in $y$. The formula is clearly true for $y=0$, since then the only nonzero term on the right side of (\ref{phi-e1}) is $\phi_m(x)\phi_0(0)=1$. Assume now that (\ref{phi-e1}) holds for some $y\in R$ and all $x\in R$. Then, by the inductive assumption, two applications of (\ref{eq:phirec}), and the convention that $\phi_{-1}=0$: 
\begin{align*}
\phi_m(x+(y+1))&=\phi_m((x+1)+y)=\sum_{i=0}^m\phi_i(x+1)\phi_{m-i}(y)\\
&= \sum_{i=0}^m \left( \phi_i(x)+\phi_{i-1}(x)\right) \phi_{m-i}(y)\\
&= \sum_{i=0}^m \phi_i(x)\left( \phi_{m-i-1}(y)+ \phi_{m-i}(y)\right) \\
&= \sum_{i=0}^m \phi_i(x)\phi_{m-i}(y+1) 
\end{align*}
as claimed.  The second identity (\ref{phi-e2}) follows by iterating  (\ref{phi-e1}).
\end{proof}

\begin{lemma}
\label{phi-higher-scales}
For $m\in[p^{k-j}]$ and $j\in\{1,\dots,k-1\}$, we have 
\begin{equation}\label{phi-higher}
\phi_{p^j m}(p^jx)=\phi_m(x).
\end{equation}
Additionally, $\phi_m(p^jx)=0$ if $p^j$ does not divide $m$.
\end{lemma}

\begin{proof}
This is an immediate consequence of (\ref{phi-digits}).

\end{proof}


\section{Discrete derivatives}\label{discrete-derivatives}

\subsection{Definitions}
A generalized polynomial of degree $m$ is expected to vanish after the successive application of $m+1$ derivatives. We prove in Lemma \ref{lemma:Dm0} that this is true for our phi functions. We start by defining the {\em degree} of a function.

\begin{definition}
For $m\in[p^k]$, define
$\Omega_m:=\mathrm{span}\{\phi_\ell :0\leq \ell \leq m\}$.
We say that 
\begin{itemize}
\item $f$ has {\em degree at most $m$} if $f\in\Omega_m$, 
\item $f$ has {\em degree equal to $m$} if $f\in\Omega_m\setminus \Omega_{m-1}$,
\item two functions $f,g$ are {\em equal up to degree $\ell$} if $f-g\in\Omega_\ell$; we write this as $f=_\ell g$.
\end{itemize}
For convenience, we set $\Omega_m:=\{0\}$ for $m<0$, so that a function $f$ has negative degree if and only if $f$ is the zero function.
\end{definition}

\begin{definition}\label{def-differences} {\bf (Discrete derivatives)}
Let $f:R \rightarrow\ZZ/p\ZZ$. We define:
\begin{equation}\label{quotient}
\begin{split}
\Delta_c f(x) &:=f(x+c)-f(x) \hbox{ for }c\in R,
\\
D_{c}f(x)& :=c^{-1}\left(f(x+c)-f(x)\right)=c^{-1}\Delta_cf(x)\hbox{ for }c\in R^\times.
\end{split}
 \end{equation}
As per our convention for functions with values in $\ZZ/p\ZZ$, 
the factor $c^{-1}$ in (\ref{quotient}) is taken to mean $(c^{-1}$ mod $p)\in \ZZ/p\ZZ$. For short, we will also write
$$
Df=D_1 f = \Delta_1 f.
$$
\end{definition}

It follows from (\ref{eq:phirec}) that
\begin{equation}\label{phi-down}
\forall m\in[p^k],\ \ 
 D\phi_m=\phi_{m-1}.
 \end{equation}
By Lemma \ref{lemma:phiLinInd}, any function
$f:R \rightarrow\ZZ/p\ZZ$ has an expansion $f=\sum c_j\phi_j$. Applying (\ref{phi-down}), we get
\begin{equation}\label{phi-down-2} 
\forall m\in[p^k],\ \ 
f\in\Omega_m\Leftrightarrow Df\in\Omega_{m-1} .
\end{equation}

\begin{lemma}\label{D-consistent} 
Let $m\in[p^k]$ and $c\in R$. Then:
\begin{itemize}
\item[(i)] $\Delta_{c}\phi_m-c\phi_{m-1}\in\Omega_{m-2}$,
\item[(ii)] If $c\in R^\times$, then $D_{c}\phi_m-\phi_{m-1} \in\Omega_{m-2}$. 
\end{itemize}
Consequently, if $f\in\Omega_m$, then $\Delta_cf\in\Omega_{m-1}$ for all $c\in R$, and $D_c f\in\Omega_{m-1}$ for all $c\in R^\times$.
\end{lemma}
\begin{proof}
If $m=0$, then $\Delta_c\phi_m=0$ for all $c\in R$ and the lemma is satisfied trivially. Assume now that $m>0$. By (\ref{phi-e1}), we have
\begin{align*} 
\phi_m(x+c)-\phi_m(x) & = \sum_{\ell=0}^{m}\phi_{m-\ell}(c)  \phi_{\ell}(x)-\phi_m(x)
\\
& =c\phi_{m-1}(x)+ \sum_{\ell=0}^{m-2}\phi_{m-\ell}(c)\phi_\ell(x),
\end{align*}
where we used that $\phi_0(c)=1$ and $\phi_1(c)=c$. This implies the lemma.
\end{proof}

\begin{lemma}
\label{lemma:Dm0}
For $m\in[p^k]$ and $f:R \rightarrow\ZZ/p\ZZ$, the following are equivalent:
\begin{itemize}
\item[(i)] $f\in\Omega_{m-1}$,
\item[(ii)] $D^m f=0$,
\item[(iii)] For any choice of $c_1,\dots,c_m\in R$ we have $\Delta_{c_m}\dots \Delta_{c_1}f= 0$.
\end{itemize}
\end{lemma}

\begin{proof}
The implication (i) $\Rightarrow$ (iii) follows by iterating Lemma \ref{D-consistent} $m$ times and using that $\Omega_{-1}=\{0\}$. 
Clearly (iii) implies (ii), by letting $c_1=\dots=c_m=1$.  

To prove that (ii) implies (i), we argue by contrapositive. Assume that $f:R\rightarrow\ZZ/p\ZZ$ has degree exceeding $m-1$. Then there is some $\ell\geq m$, a non-zero constant $c$, and some function $g$ of degree at most $\ell-1$ so that $f = c\phi_\ell+g.$ By (\ref{phi-down}), we have
\[ D^mf = c\phi_{\ell-m}+D^mg. \]
Since $D^mg\in\Omega_{\ell-1-m}$ and $c\neq 0$, it follows from linear independence of the phi functions that $D^mf$ is not the zero function.
\end{proof}

\subsection{More properties of phi functions}
\label{subsec-more-prop}

\begin{lemma}\label{digits-lemma}
Let $f:R\to\ZZ/p\ZZ$ be a function, and let $x=\sum_{j=0}^{k-1}x_jp^j$ be the $p$-adic expansion of the variable $x\in R$. Let $\ell\in\{0,1,\dots,k-1\}$
Then:

(i) $f\in\Omega_{p^\ell-1}$ if and only if 
$f(x)$ can be written as a function of the first $\ell$ digits of $x$, so that $f(x)=g(x_0,x_1,\dots,x_{\ell-1})$ for some $g:(\ZZ/p\ZZ)^\ell\to \ZZ/p\ZZ$;

(ii) $f(x) $ depends only on  $x_\ell$ (that is, $f(x)=g(x_\ell)$ for some function $g$) 
if and only if $f\in{\rm span}\{\phi_0,\phi_{p^\ell},\phi_{2p^\ell},\dots,\phi_{(p-1)p^\ell}\}$.

(iii) if $f(x)=g(x_\ell)$ for a function $g$ of degree $m\in[p]$, then $f\in{\rm span}\{\phi_0,\phi_{p^\ell},\phi_{2p^\ell},\dots,\phi_{mp^\ell}\}$.
\end{lemma}

\begin{proof} 
We prove (i), the proof of (ii) being similar.
Suppose that $f$ has degree at most $p^\ell-1$. Then $f$ is a linear combination of functions $\phi_m(x)$ with $m\leq p^\ell-1$, so that $m_i=0$ for all $i\geq \ell$. By (\ref{phi-digits}), $f$ depends only on $x_0,x_1,\dots,x_{\ell-1}$.

To prove the converse implication, we use dimension counting. There are $p^\ell$ functions $\phi_m$ with $m\leq p^\ell-1$, all linearly independent, so that $\Omega_{p^\ell-1}$ has dimension $p^\ell$.
On the other hand, the space of all functions of $x_0,x_1,\dots,x_{\ell-1}$ also has dimension $p^\ell$, since that is the number of all $\ell$-tuples $(x_0,x_1,\dots,x_{\ell-1})\in(\ZZ/p\ZZ)^\ell$. This proves (i).

For (iii), assume that $g=\phi_j$ for some $j\leq m\leq p-1$. Then 
\[ f(x) = g(x_\ell)={x_\ell\choose j} = {x \choose jp^\ell }=\phi_{jp^\ell}(x) \]
by the definition of the phi functions and by Lucas' theorem, and (iii) follows. 
\end{proof}

\begin{lemma}
\label{cor:phiprod}
Let $\ell,m\in\ZZ_{\geq0}$. Then $\phi_\ell\cdot\phi_m\in\Omega_{\ell+m}$, with
\begin{equation}\label{e-phiprod}
\phi_\ell\cdot\phi_m=_{\ell+m-1}{\ell+m\choose \ell }\phi_{\ell+m}. 
\end{equation}

\end{lemma}

We emphasize that $\phi_\ell\cdot\phi_m$ has degree \textit{at most} $\ell+m$ but not necessarily equal to it, since the coefficient of $\phi_{\ell+m}$ in (\ref{e-phiprod}) could be zero. For example, if $\ell,m\leq p^j-1$ for some $j<k$, then, by Lemma \ref{digits-lemma} (i), both $\phi_\ell(x)$ and $\phi_m(x)$ depend only on the first $j$ $p$-adic digits of $x$. Therefore so does $\phi_\ell(x)\phi_m(x)$. By Lemma \ref{digits-lemma} (i) again, $\phi_\ell\phi_m$ also has degree at most $p^j-1$, even if $\ell+m\geq p^j$.

\begin{proof}[Proof of Lemma \ref{cor:phiprod}]
We prove (\ref{e-phiprod}) by induction on $K:=\ell+m$. If $K=0$, then $\ell=m=0$ and the formula is immediate.  
Assume now that the formula is true for all $\ell,m$ with $\ell+m=K$ for some $K\geq 0$, and consider the case $\ell+m=K+1$.
Then, by (\ref{eq:phirec}) and the inductive assumption,
\begin{align*}
D(\phi_\ell\cdot\phi_m) (x)= &= \phi_\ell(x+1)\phi_m(x+1)-\phi_\ell(x)\phi_m(x)
\\
&=  (\phi_\ell(x+1)-\phi_\ell(x)) \phi_m(x+1)+\phi_\ell(x)(\phi_m(x+1)-\phi_m(x))
\\
&= \phi_{\ell-1}(x)(\phi_m(x)+\phi_{m-1}(x))+ \phi_\ell(x)\phi_{m-1}(x)
\\
& =_{\ell+m-2} \left[{\textstyle{\ell+m-1\choose \ell-1}}
+{\textstyle{\ell+m-1\choose \ell}}\right]\phi_{\ell+m-1}
\\
&= {\textstyle{\ell+m\choose \ell}}\phi_{\ell+m-1},
\end{align*}
where at the last step we used Pascal's identity. On the other hand, by (\ref{phi-down}) we also have 
$D\left({\textstyle{\ell+m\choose \ell}}\phi_{\ell+m}\right)= {\textstyle{\ell+m\choose \ell}}\phi_{\ell+m-1}$. Hence 
$$D\left(\phi_\ell\cdot\phi_m-{\textstyle{\ell+m\choose \ell }}\phi_{\ell+m}\right)
\in\Omega_{\ell+m-2},
$$
 and (\ref{e-phiprod}) follows from (\ref{phi-down-2}).
\end{proof}

\begin{lemma}
\label{bxformula}
Let $\phi_m:R\rightarrow\ZZ/p\ZZ$ and $b\in R^\times$. Then 
\[ \phi_m(bx)=_{m-1}b^m\phi_m(x)\]

\end{lemma}

\begin{proof}

We induct on $m$. The case $m=0$ is immediate, since $\phi_0$ is a constant function. 
Assume now that the result holds for $m\leq\ell$. We consider $m=\ell+1$. Let $\phi_m^b$ be the function defined by $\phi_m^b(x)=\phi_m(bx)$. Then by (\ref{phi-e1}),
\begin{align*}
D\phi_{\ell+1}^b(x)
&=\phi_{\ell+1}(bx+b)-\phi_{\ell+1}(bx)
\\
&=\sum_{j=0}^\ell \phi_{\ell+1-j}(b)\phi_j(bx) 
\\
&=_{\ell-1}b\phi_\ell(bx),
\end{align*}
where at the last step we used that $\phi_1(b)=b$. 
By the inductive hypothesis with $m=\ell$, we have
\[ 
D\phi_{\ell+1}^b(x)
=_{\ell-1}b^{\ell+1}\phi_\ell(x).\]
But we also have $D(b^{\ell+1}\phi_{\ell+1})(x)=b^{\ell+1}\phi_\ell(x)$ by (\ref{phi-down}),
so that 
$$D(b^{\ell+1}\phi_{\ell+1}-\phi_{\ell+1}^b)\in\Omega_{\ell-1}.$$
The inductive step follows from this and (\ref{phi-down-2}).
\end{proof}

The next lemma is a phi-function analogue of the fact that the coefficients of a polynomial can be computed by evaluating its derivatives at $0$.

\begin{lemma}\label{deriv-at-0}
Suppose that $f:R\to\ZZ/p\ZZ$ has the representation $f=\sum_{j=0}^{p^k-1}c_j\phi_j$. Then
\begin{equation}\label{deriv-at-0-eq}
c_\ell=D^\ell f(0)\hbox{ for all } \ell\in [p^k].
\end{equation}
\end{lemma}

\begin{proof}
By (\ref{phi-down}), we have 
$$
D^\ell f =\sum_{j=\ell}^{p^k-1}c_j\phi_{j-\ell}.
$$
We now evaluate this at $x=0$. Since $\phi_0(0)=1$ and $\phi_j(0)=0$ for all $j>0$, we get (\ref{deriv-at-0-eq}). 
\end{proof}

\begin{corollary}\label{a-expansion}
Let $a\in R$ and $m\in [p^k]$. Then
$$
\phi_m(ax)=\sum_{\ell=0}^m A_{m,\ell}(a)\, \phi_\ell(x),
$$
where $A_{m,\ell}(a)=\Delta_a^\ell \phi_m(0)$.

\end{corollary}

\begin{proof}
For $f:R\to\ZZ/p\ZZ$ and $a\in R$, define $f^a(x)=f(ax)$. Then 
$(D f^a )(x) = f(ax+a) - f(ax)=(\Delta_af)(ax)$, and, by iteration, 
\begin{equation}\label{as-boxes}
(D^\ell f^a )(x)=(\Delta^\ell_{a}f) (ax) \hbox{ for all } \ell\in [p^k].
\end{equation}
The corollary follows by applying Lemma \ref{deriv-at-0} to $f=\phi_m^a$ and then using (\ref{as-boxes}).
\end{proof}


\section{Phi functions on $R^n$}
\label{sec-multidim-phi}


\subsection{Phi functions as generalized polynomials}

For \(\alpha=(\alpha_1,\dots,\alpha_n)\in [p^k]^n\), we define $\phi_\alpha:R^n\rightarrow \ZZ/p\ZZ$ by 
\[ \phi_\alpha(x)=\phi_{\alpha_1}(x_1)\cdots\phi_{\alpha_n}(x_n). \]
Let also 
\[ \Omega_m^n:= \text{span}\{\phi_\alpha:|\alpha|\leq m\}, \]
where $|\alpha|=\sum_i\alpha_i$.  We say that a function $f:R^n\rightarrow \ZZ/p\ZZ$ has {\em degree at most $m$} if $f\in\Omega_m^n$. By convention, we set $\Omega_m^n:=\{0\}$ for $m<0$.

\begin{lemma}\label{phibasis}
The functions 
$\{\phi_\alpha:\alpha\in [p^k]^n \} $ are linearly
independent over \(\ZZ/p\ZZ\). 
\end{lemma}

\begin{proof}

We induct on \(n\). The case \(n=1\) is given by Lemma \ref{lemma:phiLinInd}.
Assume now that \(n>1\) and that the lemma holds in dimensions less than \(n\). 
Suppose that there exist \(c_\alpha\in \ZZ/p\ZZ\) such that 
\[ \sum_{\alpha\in[p^k]^n}c_\alpha \phi_\alpha(x_1,\dots,x_n)=0. \]
Write $\alpha=(\widetilde{\alpha},\alpha_n)$, where $\widetilde{\alpha}=(\alpha_1,\dots,\alpha_{n-1})$.
For fixed \(x_1,\dots,x_{n-1}\in R\), we have
\[ 0=\sum_{\alpha_n=0}^{p^k-1}\left(\sum_{\widetilde{\alpha}\in[p^k]^{n-1}}c_{(\widetilde{\alpha},\alpha_n)} \phi_{\widetilde{\alpha}}(x_1,\dots,x_{n-1})\right)\phi_{\alpha_n}(x_n). \]
This is true for all \(x_n\in R\), so by the linear independence of the functions \(\phi_{\alpha_n}\), we have
\[ \sum_{\widetilde{\alpha}\in[p^k]^{n-1}}c_{(\widetilde{\alpha},\alpha_n)} \phi_{\widetilde{\alpha}}(x_1,\dots,x_{n-1})=0 \]
for all \(\alpha_n\).
Since this holds for all \(x_1,\dots,x_{n-1}\in R\), it follows by the inductive hypothesis that \(c_{(\widetilde{\alpha},\alpha_n)}=0\) for all \(\widetilde{\alpha}, \alpha_n\). That is, \(c_\alpha=0\) for all \(\alpha\). 
\end{proof}

\begin{corollary}\label{dim-of-Omega}
For $m\leq p^k-1$, 
the dimension of $\Omega_{m}^n$ over $\ZZ/p\ZZ$ is ${m+n\choose n}$.
\end{corollary}

\begin{proof}
By Lemma \ref{phibasis}, the functions $\phi_\alpha$ with $|\alpha|\leq m$ are linearly independent. Therefore 
the dimension of $\Omega_{m}^n$ over $\ZZ/p\ZZ$ is equal to the number of $\alpha\in[p^k]^n$ such that $|\alpha|\leq m$. By Lemma~\ref{starsbars} below, this number is equal to ${m+n\choose n}$.
\end{proof}

\begin{lemma}
\label{starsbars}
Let $M,L \in\NN$ and suppose that $L<M$. Then 
\[ \#\{ (\ell_1,\dots,\ell_n)\in[M]^n:\ell_1+\dots+\ell_n\leq L\}={L+n\choose n}. \]
\end{lemma}

\begin{proof}
What we seek is equivalent to the number of $(n+1)$-tuples $(\ell_1,\dots,\ell_{n+1})\in[M]^{n+1}$ such that $\ell_1+\dots+\ell_{n+1}=L$. By a stars and bars combinatorial argument, there are ${L+n\choose n}$ such tuples.
\end{proof}

\begin{definition}\label{def-boxes}{\bf (Iterated differences)} 
Let $f:R^n\rightarrow \ZZ/p\ZZ$ be a function. For 
$r\in R^n$, define $$\Delta_r f(x):=f(x+r)-f(x).$$

The {\em iterated difference function of $f$} of order 
$d\in\NN$ with steps  $r^{(1)},\dots,r^{(d)}\in R^n$ is
\begin{equation}\label{def-eval}
\Delta_{r^{(1)}}\dots \Delta_{r^{(d)}} f(x)
= \sum_{\vec\epsilon \in\{0,1\}^{d}}(-1)^{|\epsilon|+d}f( x_{\vec\epsilon}),
\end{equation}
where for each $\vec{\epsilon} = (\epsilon_1,\dots,\epsilon_d) \in\{0,1\}^{d}$,
$$
|\epsilon|=\sum_{j=1}^d \epsilon_j, \ \ 
x_{\vec\epsilon}=x+ \sum_{j=1}^d \epsilon_j r^{(j)}.$$

We say that a function $f:R^n\to R$ is {\em $d$-null} if $\Delta_{r^{(1)}}\dots \Delta_{r^{(d)}} f(x)=0$
for all $x$ and for all $r^{(1)},\dots,r^{(d)}\in R^n$.
\end{definition}

It is useful to think of $\Delta_{r^{(1)}}\dots \Delta_{r^{(d)}} f(x)$ as the evaluation of $f$ on the $d$-dimensional box with vertices $x_{\vec\epsilon}$, where the values of $f(x_{\vec\epsilon})$ are counted with alternating $\pm$ signs as indicated. Note that the differences $r^{(j)}$ need not be distinct, in which case some of the vertices may occur in (\ref{def-eval}) for more than one value of $\vec{\epsilon}$.

In Proposition \ref{prop:phiNull} below, we show that a function $f:R^n\rightarrow \ZZ/p\ZZ$
 is $(d+1)$-null if and only if $f\in\Omega^n_d$. First, we need the following lemma.

\begin{lemma}
\label{lemma:phidiff}
For $\alpha\in[p^k]^n$ and \(r\in R^n\), we have $\Delta_r \phi_\alpha \in \Omega^n_{|\alpha|-1}$.

\end{lemma}

\begin{proof}

We induct on \(n\). When \(n=1\), the result is true by Lemma \ref{D-consistent}.
Now assume that \(n>1\), and that the lemma holds in all dimensions lower than $n$.  Then
\begin{align*} \phi_\alpha(x+r)-\phi_\alpha(x)=&\prod_{j=1}^n\phi_{\alpha_j}(x_j+r_j)-\prod_{j=1}^n\phi_{\alpha_j}(x_j)  \\
=&\big(\Delta_{r_1} \phi_{\alpha_1}(x_1)\big) \phi_{(\alpha_2,\dots,\alpha_n)}(x_2,\dots,x_n)\\
&+\phi_{\alpha_1}(x_1+r_1)\,\Delta_{(r_2,\dots,r_n)}\phi_{(\alpha_2,\dots,\alpha_n)}(x_2,\dots,x_n).
\end{align*}
By the inductive assumption, we have
\[ \Delta_{r_1} \phi_{\alpha_1}(x_1)\in\Omega^1_{\alpha_1-1},\ \ 
\Delta_{(r_2,\dots,r_n)} \phi_{(\alpha_2,\dots,\alpha_n)}
\in\Omega^{n-1}_{\alpha_2+\dots+\alpha_n-1},
\]
where we recall that $f\in\Omega_{-1}^n$ means that $f$ is the zero function. Additionally, $\phi_{\alpha_1}(x_1+r_1)\in\Omega_{\alpha_1}$ by (\ref{phi-e1}). 
It follows that $\Delta_r\phi_\alpha \in\Omega^n_{|\alpha|-1}$, as claimed.
\end{proof}

\begin{proposition}
\label{prop:phiNull}
Let \(g:R^n\rightarrow {\ZZ/p\ZZ}\) and $d\in\{0,1,\dots,n(p^k-1)\}$. Then \(g\) is \((d+1)\)-null if and only if $g\in\Omega_d^n$.
\end{proposition}
\begin{proof}
It follows by iterating Lemma \ref{lemma:phidiff} that functions in $\Omega_d^n$ are \((d+1)\)-null.
We need to prove the converse: if \(g\) is \((d+1)\)-null, then $g\in\Omega_d^n$.
We induct on \(n\), and within the induction on \(n\), induct on \(d\). 
For the base case \(n=1\), the argument below works; we just need to ignore the presence of the \(\widetilde{x}\). 

The claim is true for any $n$ when \(d=0\), since any \(1\)-null function is constant. 
Assume now that the claim is true in all dimensions lower than $n$, and true for $d$-null functions in $R^n$ for some $d\geq 0$. 

Let $g:R^n\to \ZZ/p\ZZ$ be \((d+1)\)-null. 
Define \(h(x)=g(x+e_n)-g(x-e)\), where \(e_n\) is the vector with a \(1\) in the \(n^{th}\) coordinate, and \(0s\) otherwise. 
For $x\in R^n$, we will write \(x=(x_1,\dots,x_n)=: (\widetilde{x},x_n)\), where
 \(\widetilde{x}=(x_1,\dots,x_{n-1})\). Then
\begin{equation}
g(x)= g(\widetilde{x},0)+h(\widetilde{x},0)+h(\widetilde{x},1)+\dots + h(\widetilde{x},x_n-1).
\end{equation}
The function $\widetilde{g}(\widetilde{x})=g(\widetilde{x},0)$, considered as a function on $R^{n-1}_{\widetilde{x}}$, is 
\((d+1)\)-null. By the lower-dimensional part of the inductive assumption, we have $\widetilde{g}\in\Omega^{n-1}_d$.
Considering now $g(\widetilde{x},0)$ as a function of $n$ variables that is constant in the $e_n$ direction, we have 
$$g(\widetilde{x},0)= \widetilde{g}(\widetilde{x})\phi_0(x_n)\in\Omega^n_d.$$
Next, $h=\Delta_{e_n}g$ is $d$-null. By the inductive hypothesis on \(d\), we can write
\[ h(x)=\sum_{\beta\in [p^k]^{n-1},|\beta|\leq d-1}\phi_\beta(\widetilde{x})\sum_{j=0}^{d-1-|\beta|}a_{\beta,j}\phi_j(x_n).  \] 
so that
\begin{align*}
\sum_{\ell=1}^{x_n}h(\widetilde{x},\ell)
&= \sum_{\ell=1}^{x_n}\sum_{|\beta|\leq d-1}\phi_\beta(\widetilde{x})\sum_{j=0}^{d-1-|\beta|}a_{\beta,j}\phi_j(\ell) \\
&=\sum_{|\beta|\leq d-1}\phi_\beta(\widetilde{x})\sum_{j=0}^{d-1-|\beta|}a_{\beta,j}\sum_{\ell=1}^{x_n}\phi_j(\ell) \\
&=\sum_{|\beta|\leq d-1}\sum_{j=0}^{d-1-|\beta|}a_{\beta,j}\phi_\beta(\widetilde{x})\phi_{j+1}(x_n). \\
\end{align*}
But for each pair \(\beta,j\) appearing in the sum, 
\[ \phi_\beta(\widetilde{x})\phi_{j+1}(x_n) =\phi_{(\beta,j+1)}(x) \in \Omega^n_d.\]
This ends the proof of the proposition.
\end{proof}

\begin{corollary}
\label{n-dull}
Let \(g:R^n\rightarrow {\ZZ/p\ZZ}\) and $d\in\{0,1,\dots,p-1\}$. Then $g\in\Omega_d^n$ if and only if $g$ is a polynomial in $R[x_1,\dots,x_n]$ of degree at most $d$.
\end{corollary}

In particular, since $\{\phi_\alpha:\ |\alpha|\leq d\}$ is a basis for $\Omega_d^n$, it follows that each $\phi_\alpha$ with $|\alpha|\leq p-1$ is a polynomial of degree $|\alpha|$.

\begin{proof}
It is well known, and easy to check directly, that if $f$ is a polynomial of degree $d$ then $\Delta_c f$ is a polynomial of degree at most $d-1$ for any $c\in R$. By iteration, it follows that every polynomial $g$ of degree $d\leq p-1$ is \((d+1)\)-null. By Proposition \ref{prop:phiNull}, we have $g\in \Omega_d^n$. Moreover, $\Omega_d^n$ and the space of all polynomials in $R[x_1,\dots,x_n]$ of degree at most $d$ have the same dimension ${d+n\choose n}$ (the number of distinct multiindices $\alpha=(\alpha_1,\dots,\alpha_n)$ with $|\alpha|\leq d$; see Lemma \ref{starsbars} above). Therefore the two spaces are equal.
\end{proof}



\subsection{Phi functions and hyperplanes}\label{phi-and-hyperplanes}

Recall that
\begin{equation}\label{hyper-span}
\calh^n:=\mathrm{span}\{\one_{H_b(a)}:\ a \in R^n, b\in\mathbb{P}R^{n-1}\}
\end{equation}
is the linear span of indicator functions of affine hyperplanes. We will refer to functions in $\calh^n$ as \emph{hyperplane functions} in $R^n$. 

We are interested in characterizing hyperplane functions and, in particular, determining the dimension of $\calh^n$. To this end, we first find a  spanning set in terms of the phi functions.

\begin{lemma}\label{lemma-move-to-phi}
\label{hyperasphi}
We have
\[\calh^n=\emph{span}\{\phi_\ell(\langle x,b\rangle): \ell\in[p^k],\   b\in\mathbb{P}R^{n-1}      \}.\]
\end{lemma}

\begin{proof}
It suffices to prove that for each $b\in\PP R^{n-1}$,
\begin{equation}\label{move-to-phi}
\begin{split}
\mathrm{span}\{\one_{H_b(a)}: a\in R\}& =\mathrm{span}\{f(\langle x,b\rangle): f\in (\ZZ/p\ZZ)^R\}
\\
&=\mathrm{span}\{\phi_\ell(\langle x,b\rangle): \ell\in[p^k]\}.
\end{split}
\end{equation}
The second equality in (\ref{move-to-phi}) follows from Lemma \ref{lemma:phiLinInd}. We now prove the first one. 
For any $b\in\PP R^{n-1}$ and $a\in R$, we may write
$$
 \one_{H_b(a)}(x)=\one_{\{0\}}(\langle x-a, b\rangle)
=\one_{\{\langle a,b\rangle\}} (\langle x, b\rangle)
$$  
which shows that $\one_{H_b(a)}$ can be written as a single-variable function of $\langle x, b\rangle$ as claimed. Conversely, let $f:R\to\ZZ/p\ZZ$ be a function. Then
$$
f(x)=\sum_{c\in R}f(c)\one_{\{c\}}, \hbox{ hence } f(\langle x,b\rangle)= \sum_{c\in R}f(c)\one_{\{c\}} (\langle x, b\rangle)
$$
Since $b\in\PP R^{n-1}$, there exists $i\in\{1,2,\dots,n\}$ such that $b_i$ is invertible. For each $c\in R$, let $\overline{c}\in R^n$ be the vector whose $i$-th coordinate is $cb_i^{-1}$ and all other coordinates are $0$. Then 
$\langle \overline{c},b\rangle=c$, so that
$$
\one_{\{c\}} (\langle x, b\rangle) =\one_{H_b(\overline{c})}(x).
$$
Hence every function $f(\langle x,b\rangle)$ can be written as a linear combination of hyperplane functions with the normal vector $b$. This ends the proof of (\ref{move-to-phi}), and of the lemma.

\end{proof}

\begin{proposition}\label{hyper-to-phi}
We have $\calh^n\subset \Omega^n_{p^k-1}$
for all $k\geq 1$. Moreover, if $k=1$ then $\calh^n= \Omega^n_{p-1}$.
In particular, 
\begin{equation}\label{Hbound-trivial}
\hbox{\rm rank}(\mathcal{A}_{p^k,n}^*)=\hbox{ \rm dim}(\calh^n )\leq {p^k-1+n\choose n},
\end{equation}
and (\ref{Hbound-trivial}) holds with equality when $k=1$.
\end{proposition}

\begin{proof}[Proof of Proposition \ref{hyper-to-phi}, part 1]
We prove that $\calh\subset \Omega^n_{p^k-1}$
for all $k\geq 1$. By Lemma \ref{lemma-move-to-phi}, it suffices to prove that $\phi_\ell(\langle b,x\rangle)\in\Omega_\ell^n$ for all $\ell\in[p^k]$ and $b\in\mathbb{P}R^{n-1}$. To this end, we use (\ref{phi-e1}) to write
\begin{equation}\label{e-hyperterms}
\phi_\ell(\langle b,x\rangle)= \sum_{|\alpha|\leq\ell} \phi_{\alpha_1}(b_1x_1)\cdots\phi_{\alpha_n}(b_nx_n). 
\end{equation}
But Lemma~\ref{bxformula} implies that
\( \phi_{\alpha_i}(b_i\,\cdot \,)\in \Omega_{\alpha_i}\) for each $i$.
Hence each term on the right side of (\ref{e-hyperterms}) has degree at most $|\alpha|$, which in turn implies that $\phi_\ell(\langle b,x\rangle)\in\Omega_\ell^n$ as claimed. The bound (\ref{Hbound-trivial}) follows from Corollary \ref{dim-of-Omega} with $m=p^k-1$.
\end{proof}

The proof of the converse inclusion for $k=1$ will be based on the two lemmas below. 
For $d\in[p]$, let
 \(\calp^n_{\leq d}:= (\ZZ/p\ZZ)[x_1,\dots,x_n]_{\leq d}\) be the space of polynomials in $n$ variables of degree at most \(d\) over $\ZZ/p\ZZ$, and let \(\calp^n_{= d} \) be the subspace of homogeneous, degree \(d\) polynomials in
\(\calp^n_{\leq d} \). 
The relation between polynomials and hyperplane indicator functions for $k=1$ is well understood in the literature, see \cite{GD, MM, Smith}. The proof below is provided for completeness.

\begin{lemma}
\label{lemma:spanHomPolys}
Let $k=1$. For any \(n\in\NN\) and any \(d\in\{0,1,\dots,p-1\}\), 
\[ \emph{span}\left\{\langle x,b\rangle^d:b\in \PP(\ZZ/p\ZZ)^n\right\}=\calp^n_{=d}. \]

\end{lemma}

\begin{proof}
We proceed with induction on \(n\). The case \(n=1\) is immediate. Suppose that the statement holds in all dimensions lower than $n$. 
We show that
\[ x^\alpha\in \text{span}\left\{\langle x,b\rangle^d:b\in \PP(\ZZ/p\ZZ)^{n-1}\right\} \]
for all \(\alpha\in [p]^{n}\) with \(|\alpha|=d\). 

For $x\in (\ZZ/p\ZZ)^n$, we write $x=(\widetilde{x},x_n)$, where \(\widetilde{x}=(x_1,\dots,x_{n-1})\). Write also \(\alpha=(\beta,\alpha_{n})\), where \(\beta=(\alpha_1,\dots,\alpha_{n-1})\), so that \(x^\alpha=x_1^{\alpha_1}\cdots x_n^{\alpha_n}= \widetilde{x}^{\beta}x_{k+1}^{\alpha_{k+1}}\).
Let \(\ell=|\beta|\). By the inductive hypothesis, 
we may write
\[ \widetilde{x}^{\beta}x_{n}^{\alpha_{n}}=\sum_{c\in\PP(\ZZ/p\ZZ)^{n-2} }a_c\langle\widetilde{x},c\rangle^\ell x_{n}^{\alpha_{n}}. \]
Therefore it suffices to show that
\[ \langle\widetilde{x},c\rangle^\ell x_{n}^{\alpha_{n}}\in \text{span}\left\{\langle x,b\rangle^d:b\in \PP(\ZZ/p\ZZ)^{n-1}\right\} \]
for all  \(c\in\PP(\ZZ/p\ZZ)^{n-2}\).
To this end, it is enough to prove that
\begin{equation}\label{hom-e1}
\left\{\langle\widetilde{x},c\rangle^jx_{n}^{d-j}:j=0,1,\dots,d-1\right\}\subset \text{span}\left\{\langle x,(c,i)\rangle^d-(ix_{n})^d:i=1,\dots,d\right\}, 
\end{equation}
where $(c,i)=(c_1,\dots,c_{n-1},i)\in\PP(\ZZ/p\ZZ)^{n-1}.$
Note that
\[ \langle x,(c,i)\rangle^d-(ix_{n})^d=\sum_{j=0}^{d-1}{d\choose j}i^{j}\langle\widetilde{x},c\rangle^{d-j}x_{n}^j. \]
We consider this as a system of $d$ linear equations with \(\langle\widetilde{x},c\rangle^{d-j}x_{n}^{j}\). The coefficient matrix of this system has the determinant
\[ \left(\prod_{j=0}^{d-1}{d\choose j}\right)\det\begin{pmatrix}
1 & 1 &  1 & \cdots & 1 \\
1 & 2 & 2^2 &\cdots & 2^{d-1} \\
\vdots \\
1 & d & d^2 &\cdots & d^{d-1} 
\end{pmatrix}
= \prod_{j=0}^{d-1}{d\choose j}\prod_{1\leq i<j\leq d}(i-j),
\]
where we evaluated the determinant of the Vandermonde matrix.
Since \({d\choose j}\neq0\) for \(d\leq p-1\), 
our coefficient matrix is nonsingular, so that we can solve for 
\(\langle\widetilde{x},c\rangle^{d-j}x_{n}^{j}\) as claimed in (\ref{hom-e1}).

\end{proof}

\begin{lemma}
\label{lemma:spanPolys}
Let $k=1$. For any \(d\in\{0,1,\dots,p-1\}\), we have
\[ \emph{span}\{\langle x-a,b\rangle^d: b\in\PP(\ZZ/p\ZZ)^{n-1},a\in(\ZZ/p\ZZ)^n\}=\calp^n_{\leq d}. \]
\end{lemma}

\begin{proof}

By Lemma~\ref{lemma:spanHomPolys}, it suffices to show that
\[ \text{span}\{\langle x-a,b\rangle^d:a\in(\ZZ/p\ZZ)^n\}=\text{span}\{\langle x,b\rangle^\ell:\ell\in\NN,\; \ell\leq d \}. \]
For any \(a\not\in H_b\), we know that \(\langle a,b\rangle\) is non-zero, and so a unit. Then \(\langle ca,b\rangle\) will range over all values in \(\ZZ/p\ZZ\) as \(c\) ranges over all values in \(\ZZ/p\ZZ\). Consequently, 
\[ \{\langle x-a,b\rangle^d:a\in(\ZZ/p\ZZ)^n\}=\{(\langle x,b\rangle-c)^d:c=0,\dots,p-1\}. \]
Consider the system of equations 
\[ (\langle x,b\rangle-c)^d = \sum_{j=0}^d{d\choose j}c^j\langle x,b\rangle^{d-j}, \quad c=0,\dots,d,\]
with \(\langle x,b\rangle^{d-j}\) as the unknowns. The coefficient matrix of this system has the determinant
\[ \left(\prod_{j=0}^{d}{d\choose j}\right)\det\begin{pmatrix}
1& 0 & 0 &\cdots &0 \\
1 & 1 &  1 & \cdots & 1 \\
1 & 2 & 2^2 &\cdots & 2^{d} \\
\vdots \\
1 & d & d^2 &\cdots & d^{d} 
\end{pmatrix}=\prod_{j=0}^{d}{d\choose j}\prod_{c=2}^d c\prod_{1\leq u<v\leq d}(u-v).\]
This is non-zero as in the proof of Lemma \ref{lemma:spanHomPolys}, hence we can solve for \(\langle x,b\rangle^{d-j}\).

\end{proof}

\begin{proof}[Proof of Proposition \ref{hyper-to-phi}, part 2]
Assume that $k=1$. Observe that the characteristic function of a hyperplane $H_b(a)$ may be written as
$\one_{H_b(a)}(x)=1- \langle x-a , b\rangle^{p-1}$ mod $p$. 
By Lemma \ref{lemma:spanPolys} with $d=p-1$, we have $\calh^n=\calp^n_{\leq p-1}$. It follows by Corollary \ref{n-dull} that $\calh^n= \Omega^n_{p-1}$, as claimed.

\end{proof}

\begin{remark}\label{homog-remark}
Let $\mathcal{H}^n_0=\Span\{\one_{H_b}:\ b\in \PP(\ZZ/p\ZZ)^{n-1}\}$ be the span of homogeneous hyperplane functions. The same argument as above, but using Lemma~\ref{lemma:spanHomPolys} instead of \ref{lemma:spanPolys}, shows that $\mathcal{H}^n_0$ is spanned by homogeneous polynomials of degree $p-1$ together with $\one_{(\ZZ/p\ZZ)^n}$, the function identically equal to 1. To prove the converse, it suffices to verify that $\one_{(\ZZ/p\ZZ)^n}$ can be represented as a linear combination of hyperplane functions. Such representation is provided by
$$
\one_{(\ZZ/p\ZZ)^n}=\one_{x_1=0}+\sum_{c=0}^{p-1}\one_{x_2=cx_1}.
$$
This offers a proof of Theorem \ref{k-is-1}.

\end{remark}




\section{Degree lowering for products}
\label{sec-degree-lowering}

\begin{lemma}\label{degree-loss}
Let $f(x,y)=\phi_m(x_iy_j)$ for some $m\in[p^k]$ and $i,j\geq 1$, where $x=\sum x_\ell p^\ell$ and $y=\sum y_\ell p^\ell$
are the $p$-adic expansions of $x,y\in R$.
Then $f$
has degree at most $mp^{i+j}$, with equality attained only when $p=2$ and $i=j=m=1$.
\end{lemma}

\begin{proof}
By Lemma \ref{digits-lemma} (ii), we have
$$
f(x,y)=\sum_\alpha c_\alpha \phi_{\alpha_1}(x)\phi_{\alpha_2}(y),
$$
where the summation is over $\alpha=(\alpha_1,\alpha_2)$ with
$$
\alpha_1\in\{0,p^i,2p^i,\dots,(p-1)p^i\},\ 
\alpha_2\in\{0,p^j,2p^j,\dots,(p-1)p^j\}.
$$
Thus the combined degree of each $\phi_{\alpha_1}(x)\phi_{\alpha_2}(y)$ is at most 
$(p-1)p^i+(p-1)p^j= p^{i+1}+p^{j+1}-p^i-p^j$. We may assume that $i\leq j$.

\begin{itemize}
\item If $i<j$, then $p^{i+1}\leq p^j$, so that $p^{i+1}+p^{j+1}-p^i-p^j\leq p^{j+1}-p^i<p^{i+j}$.

\smallskip

\item If $i=j\geq 2$, then $2p^{i+1}-2 p^i <2p^{i+1}\leq p^{i+2}\leq p^{i+j}$.

\smallskip

\item If $i=j=1$, then $2p^2-2p<2p^2=2p^{i+j}$. This is at most $mp^{i+j}$ unless $m=1$. However, if $m=1$, then
$$
\phi_1(x_1y_1)=x_1y_1=\phi_{p}(x)\phi_{p}(y)
$$
has degree $2p\leq p^{2}$, with equality only when $p=2$.

\end{itemize}

\end{proof}

Our next goal is to determine the degree of $f(x,y)=\phi_m(xy)$ as a function of 2 variables 
for $m\in[p^k]$.
Recall from Corollary \ref{a-expansion} that
\begin{equation}\label{lower-e1}
\phi_m(xy)=\sum_{\ell=0}^m A_{m,\ell}(y)\, \phi_\ell(x),
\end{equation}
where 
$A_{m,\ell}(y)=\Delta_{y}^\ell \phi_m(0)=\sum_{i=0}^\ell (-1)^{i+\ell} {\ell\choose i} \phi_m(iy).$
By Lemma \ref{bxformula}, $\phi_m(iy)$ is a function of degree at most $m$ in $y$ for each $i$. Hence $\phi_m(xy)$ 
has degree at most $m$ in each variable separately.

We will see below that the \textit{combined} degree of $\phi_m(xy)$, considered as a function of two variables, cannot be much larger than $m$. This is in sharp contrast to polynomials over $\ZZ$, where the combined degree of $(xy)^m=x^my^m$ is always $2m$. 

\begin{proposition}\label{product-degree}
Let $f(x,y)=\phi_m(xy)$ for some $m\in[p^k]$ and $x,y\in R$.
Then $f$ has degree at most $m+2(p-1)$. Specifically, we have
\begin{equation}\label{h-expand-1}
\phi_m(xy) 
=\sum_{\alpha}c_{m,\alpha}\, \phi_{\alpha_1}(x)\phi_{\alpha_2}(y),
\end{equation}
where the coefficients $c_{m,\alpha}$ satisfy $c_{m,\alpha}=0$ if $|\alpha|>m+2(p-1)$.

\end{proposition}

\begin{proof}
Let $m\in[p^k]$, and let $x=\sum x_ip^i$ and $y=\sum y_ip^i$ be the $p$-adic expansions of $x,y\in R$.
By (\ref{phi-e1}) and Lemma \ref{phi-higher-scales}, we have
\begin{align*}
\phi_m(xy)&= \phi_m\left(\sum_{i+j\leq k-1} p^{i+j}x_iy_j\right)
\\
&= \sum_{\vec{m}}\prod_{i,j}\phi_{m_{ij} }(x_iy_j),
\end{align*}
where the summation is over all $\vec{m}=(m_{ij})_{i+j\leq k-1} $ such that $\sum_{i,j} m_{ij}p^{i+j}=m$.
Fix $\vec{m}$, and consider the corresponding term in the sum above:
\begin{align*}
 \prod_{i,j}\phi_{m_{ij} }(x_iy_j) 
&= \left( \prod_{j=0}^{k-1} \phi_{m_{0j}} (x_0y_j)\right) 
 \left( \prod_{i=0}^{k-1} \phi_{m_{i0}} (x_iy_0)\right)
  \left( \prod_{i,j\geq 1} \phi_{m_{ij}} (x_iy_j)\right) 
 \\[2ex]
&=:P_1P_2P_3,
\end{align*}
By Lemma \ref{degree-loss}, $P_3$ has degree at most
\begin{equation}\label{p3-eq}
\sum_{i,j\geq 1} m_{ij}p^{i+j}.
\end{equation}
Next, we consider $P_1$. By (\ref{lower-e1}) and Lemma \ref{digits-lemma}, each factor $\phi_{m_{0j}} (x_0y_j)$ has degree at most $p-1$ in $x$ and at most $m_{0j}$ in $y_j$, therefore at most $m_{0j}p^j$ in $y$. In other words, we can write  $\phi_{m_{0j}} (x_0y_j)$ as a linear combination of terms of the form $\phi_{\beta_1}(x_0)\phi_{\beta_2}(y)$, where $\beta_2\leq m_{0j}p^j$. Taking the product, and applying Lemma \ref{digits-lemma} to the factors involving $x_0$ and Lemma 
\ref{cor:phiprod} to the factors involving $y$, we see that $P_1$ has degree at most
\begin{equation}\label{p1-eq}
(p-1)+ \sum_j m_{0j}p^j.
\end{equation}
Similarly, $P_2$ has degree at most $(p-1)+ \sum_i m_{i0}p^i$. Combining this with (\ref{p3-eq}) and (\ref{p1-eq}), we get the desired bound.
\end{proof}


\section{An upper bound on the rank of hyperplane functions}
\label{sec-main-upper-bound-1}

In this section we prove our lower bound on the rank of the reduced point-affine hyperplane incidence matrix, which we state again for the reader's convenience.

\begin{theorem}\label{main-upper-bound}
Let $p$ be prime, and let $k,n\in\NN$. Then
\begin{equation} 
\label{upperbound}
\emph{rank}(\mathcal{A}_{p^k,n}^*)\leq (2n){\lfloor p^k/2\rfloor+(n-1)(p-1)+n\choose n}.
\end{equation}
\end{theorem}

Before starting the proof of the theorem, we compare (\ref{upperbound}) to the upper bound ${p^k-1+n\choose n}$ given by (\ref{Hbound-trivial}).
Suppose that $n$ is small relative to $p^{k-1}$, with $n<\epsilon p^{k-1}$ for some $\epsilon>0$. Then
\[ (2n){\lfloor p^k/2\rfloor+(n-1)(p-1)+n\choose n} \leq \frac{(p^k+2(n-1)(p-1)+2n)^n}{2^{n-1}(n-1)!}<\frac{p^{kn}(1+4\epsilon)^n}{2^{n-1}(n-1)!}. \]
Meanwhile, we have 
\[ {p^k-1+n\choose n}\geq \frac{p^{kn}}{n!}. \]
Hence, for $n<\epsilon p^{k-1}$, the estimate in~(\ref{upperbound}) improves on that in Proposition \ref{Hbound-trivial}  by a factor of at least $n2^{-(n-1)}(1+4\epsilon)^n$.

\begin{proof}[Proof of Theorem \ref{main-upper-bound}]
Recall that the rows of $\mathcal{A}^*_{p^k,n}$ are given by indicator functions of hyperplanes $H_b(a)$ with $a\in R^n$ and $b\in\mathbb{P}R^{n-1}$. Hence its rank is equal to the dimension of $\calh^n$ over $\ZZ/p\ZZ$, where $\calh^n$ was defined in (\ref{hyper-span}). By Lemma \ref{lemma-move-to-phi}, we further have
\begin{equation}\label{phi-span}
\calh^n={\rm span}\{\phi_\ell(\langle x,b\rangle): \ell\in[p^k],\   b\in\mathbb{P}R^{n-1}      \}.
\end{equation}
Any $b\in\PP R^{n-1}$ has a representative in $R^n$ with at least one component equal to $1$. Hence
\begin{equation} 
\label{A-to-H}
\text{rank}(\mathcal{A}_{p^k,n}^*)\leq n\cdot  \text{rank}(\mathbb{H}^{(n)}),
\end{equation}
where  $\mathbb{H}^{(n)}$ is the matrix with rows indexed by $(m,\widetilde{a})\in[p^k]\times R^{n-1}$, columns indexed by $x=(\widetilde{x},x_n)\in R^n$, and entries
\[ \mathbb{H}^{(n)}_{(m,\widetilde{a}),x}=\phi_m(\langle \widetilde{a},\widetilde{x}\rangle +x_n). \]

Let $\widetilde{a}=(a_1,\dots,a_{n-1})\in R^{n-1}$ and $m\in[p^k]$. 
By (\ref{phi-e1}) and then Proposition \ref{product-degree}, we have 
\begin{equation}\label{long-gamma}
\begin{split}
\phi_m(\langle \widetilde{a},\widetilde{x}\rangle+ x_n)
&= \sum_{\ell_1+\dots+\ell_{n-1}+\beta_n=m}\phi_{\ell_1}(a_1x_1)\cdots \phi_{\ell_{n-1}}(a_{n-1}x_{n-1})\phi_{\beta_n}(x_n) \\
&= \sum_{\ell_1+\dots+\ell_{n-1}+\beta_n=m}\sum_{\widetilde{\alpha},\widetilde{\beta}}
\gamma(\widetilde{\ell},\widetilde{\alpha},\widetilde{\beta}) \phi_{\widetilde{\alpha}}(\widetilde{a})\phi_\beta(x),
%
\end{split}
\end{equation}
where we write 
\begin{align*}
\widetilde{\alpha}&=(\alpha_1,\dots,\alpha_{n-1})\in[p^k]^{n-1},
\\
\beta&=(\widetilde{\beta},\beta_n)=(\beta_1,\dots,\beta_{n})\in [p^k]^{n},
\\
\widetilde{\ell}&=(\ell_1,\dots,\ell_{n-1})\in [p^k]^{n-1},
\end{align*}
and 
\begin{equation}\label{gamma-coeff}
\gamma(\widetilde{\ell},\widetilde{\alpha},\widetilde{\beta}) 
=\prod_{j=1}^{n-1} c_{\ell_j,(\alpha_j,\beta_j)},
\end{equation}
where $c_{\ell_j,(\alpha_j,\beta_j)}$ are the coefficients in the expansion (\ref{h-expand-1}).

Let $\Phi$ be the matrix with rows indexed by $\beta\in[p^k]^n$, columns indexed by $x\in R^n$, and entries 
$\Phi_{\beta,x}=\phi_\beta(x)$. Let also $\Psi$ be the block-diagonal matrix with rows indexed by $(m,\widetilde{a})\in[p^k]^n$, columns indexed by $(\mu,\widetilde{\alpha})\in R^n$, and entries 
$$
\Psi_{(m,\widetilde{a}),(\mu,\widetilde{\alpha})}={\bf 1}_{m=\mu}\phi_{\widetilde{\alpha}}(\widetilde{a}).
$$
Then (\ref{long-gamma}) can be written in matrix form as
\[ \mathbb{H}^{(n)}= \Psi \bbB^{(n)} \Phi, \]
where $\bbB^{(n)}$ is the matrix with rows indexed by $(m,\widetilde{\alpha})\in R^n$,
columns indexed by $\beta\in[p^k]^n$, and entries 
$$
\bbB^{(n)}_{(m,\widetilde{\alpha}),\beta}= \sum_{\ell_1+\dots+\ell_{n-1}+\beta_n=m} 
\gamma(\widetilde{\ell},\widetilde{\alpha},\widetilde{\beta}) .
$$
Since both $\Phi$ and $\Psi$ are nonsingular by Lemma~\ref{phibasis}, it follows that $\mathbb{H}^{(n)}$ and $\bbB$ have the same rank. The next proposition completes the proof of Theorem \ref{main-upper-bound}.
\end{proof}


\begin{proposition}
\label{UBn}
We have 
\[ \emph{rank}\left(\mathbb{B}^{(n)}\right)\leq  2{\lfloor p^k/2\rfloor+(n-1)(p-1)+n\choose n}.\] 
\end{proposition}

\begin{proof}
We claim that $\bbB^{(n)}_{(m,\widetilde{\alpha}),\beta}= 0$ for all $m,\widetilde{\alpha},\beta$ such that
\begin{equation}\label{zeroterm}
\sum_{j=1}^{n-1}\alpha_j + \sum_{j=1}^{n}\beta_j > m+2(n-1)(p-1).
\end{equation}
Indeed, assume that $m,\widetilde{\alpha},\beta$ satisfy (\ref{zeroterm}), and consider a contributing term 
$$\gamma(\widetilde{\ell},\widetilde{\alpha},\widetilde{\beta}) =\prod_{j=1}^{n-1} c_{\ell_j,(\alpha_j,\beta_j)}
\hbox{ with }\ell_1+\dots+\ell_{n-1}+\beta_n=m.
$$
By (\ref{zeroterm}), we have
$$
\sum_{j=1}^{n-1}(\alpha_j + \beta_j )+\beta_n > \sum_{j=1}^{n-1}\ell_j +\beta_n + 2(n-1)(p-1).
$$
Hence there is at least one $j$ such that $\alpha_j + \beta_j >\ell_j+2(p-1)$. By Proposition \ref{product-degree}, we have $c_{\ell_j,(\alpha_j,\beta_j)}=0$ for that $j$, so that $\gamma(\widetilde{\ell},\widetilde{\alpha},\widetilde{\beta}) =0$.
Since this is true for all contributing terms, the claim follows.

Write $|\widetilde{\alpha}|=\sum_{j=1}^{n-1}\alpha_j$ and $|\beta|=\sum_{j=1}^{n}\beta_j $ for short. 
We choose $\lambda\in[p^k]$, to be determined, and decompose 
$\mathbb{B}^{(n)}$ into two matrices, $\mathbb{B}_{\leq \lambda}^{(n)}$ and $\mathbb{B}^{(n)}_{>\lambda}$, with rows and columns indexed as for $\mathbb{B}^{(n)}$. Let $\mathbb{B}^{(n)}_{\leq\lambda}$ be defined so that for any row indexed by $(m,\widetilde{\alpha})$ with $m-|\widetilde{\alpha}|\leq \lambda$, the $(m,\widetilde{\alpha})$-row of $\mathbb{B}^{(n)}_{\leq\lambda}$ matches the $(m,\widetilde{\alpha})$-row of $\mathbb{B}^{(n)}$. All other rows are zero. Then define $\mathbb{B}^{(n)}_{>\lambda}$ so that
\begin{equation}
\label{matrixsum}
\mathbb{B}^{(n)}=\mathbb{B}^{(n)}_{\leq\lambda}+\mathbb{B}^{(n)}_{>\lambda}.
\end{equation}

First consider $\mathbb{B}^{(n)}_{\leq\lambda}$. All its non-zero entries lie in rows indexed by $(m,\widetilde{\alpha})$ with $m-|\widetilde{\alpha}|\leq\lambda$. By~(\ref{zeroterm}), any column indexed by $\beta$ satisfying $|\beta|>\lambda+2(n-1)(p-1)$ is the zero vector.
Thus bounding the rank of the matrix by its number of non-zero columns, we obtain
\begin{align*}
\text{rank}(\mathbb{B}^{(n)}_{\leq\lambda})&\leq \#\{\beta\in[p^k]^n:|\beta|\leq \lambda+2(n-1)(p-1)\}
\\
&={\lambda+2(n-1)(p-1)+n\choose n}
\end{align*}
by Lemma~\ref{starsbars}.

Now we consider $\mathbb{B}^{(n)}_{>\lambda}$; for this, we bound the rank of the matrix by its number of non-zero rows: 
\begin{align*}
\text{rank}(\mathbb{B}^{(n)}_{>\lambda})&\leq \#\{(m,\widetilde{\alpha})\in[p^k]\times[p^k]^{n-1}:m-
|\widetilde{\alpha}|>\lambda\} \\
&= \#\{(m,\widetilde{\alpha})\in[p^k]\times[p^k]^{n-1}:(p^k-1-m)+|\widetilde{\alpha}|<p^k-1-\lambda\} \\
&= {p^k-\lambda-2+n\choose n}
\end{align*}
by Lemma~\ref{starsbars} applied with $\ell_1=p^k-1-m$ and $\ell_i=\alpha_i$ for $i>1$. 

Taking $\lambda=\lfloor p^k/2\rfloor-(n-1)(p-1)$, and applying the subadditivity of rank to~(\ref{matrixsum}), we see that
\begin{align*} 
\text{rank}(\mathbb{B}^{(n)})&\leq {\lfloor p^k/2\rfloor+(n-1)(p-1)+n\choose n}+{p^k-\lfloor p^k/2\rfloor+(n-1)(p-1)-2+n\choose n} \\
&\leq 2\cdot{\lfloor p^k/2\rfloor+(n-1)(p-1)+n\choose n}.
\end{align*}

\end{proof}






\section{Geometric test for hyperplane functions}
\label{sec-fan}

Theorem~\ref{main-upper-bound} shows that, in general, the linear span of affine hyperplane functions is strictly smaller than the span of phi functions of degree less than $p^k$. 
In this section we develop a geometric test for determining which phi functions are not in the span of hyperplane functions. In Subsection~\ref{subsec-geo-test-2} we prove a specific case of the test in dimension $n=2$, and then use it to show that a particular phi function is not in the span of hyperplane functions. Afterwards, we prove the test in generality.  The full result is given in Theorem~\ref{thm-geotest}.

Recall that $R=\ZZ/p^k\ZZ$. To simplify the multiscale notation below, we will also write $R_\ell =\ZZ/p^\ell\ZZ$ for $1\leq \ell\leq k$, so that $R_k=R$ and $R_1=\ZZ/p\ZZ$. 
A {\em line} in a direction $b\in R^n$ is a set of the form
$$
L_b(a) =\{a + t b:\ t\in R\} \hbox{ for some }a\in R^n.
$$
If $b$ is nondegenerate, $L_b(a) $ has $|R|=p^k$ distinct elements.

In $R^n$, we define the \textit{canonical directions} to be elements of the set $\mathcal{B}=\bigcup_{i=1}^n\mathcal{B}_i$ where 
\[ \mathcal{B}_i=\{(p\ell_1,\dots,p\ell_{i-1},1,\ell_{i+1},\dots,\ell_n):\ell_i\in R \}.\]
Then any line $L\subset R^n$ may be written in the form $\{a+tb:t\in R\}$ for a unique direction vector $b\in\mathcal{B}$. Henceforth, when we refer to the direction of a line, this direction is an element of $\mathcal{B}$. For $b,b'\in\mathcal{B}$, we define the \textit{$p$-adic angle} between $b$ and $b'$ to be $\angle(b,b')=p^{-s}$ where $p^s\parallel (b-b')$.
If $L$ and $L'$ are lines with directions $b$ and $b'$ respectively, we define the angle between them to be
$\angle(L,L')=\angle(b,b')$.

For $0\leq \ell\leq k$, define the projection map $\pi_{\ell}:R^n\rightarrow R_\ell^n$ by
\[ \pi_{\ell}(x)=x\bmod{p^\ell}.\]
Clearly, the mappings $\pi_{\ell}$ are linear.
For $0\leq \ell\leq k$, define a \textit{cube on scale $\ell$} to be a set of the form 
\[ Q=Q_\ell(x)=\{y\in R^n:\pi_\ell(y)=\pi_\ell(x)\}\subset R^n \]
for a fixed $x\in R^n$. In dimension $n=2$, we refer to $Q$ as a square. Note that a cube on scale $0$ is the entire $R^n$, and a cube on scale $k$ is a single point.

Next, we will define a type of set that we call a \textit{fan.} Our geometric test will show that hyperplane functions are orthogonal to characteristic functions of fans. 

\begin{definition}
\label{def-fan}
{\bf (Fans in dimension $n=2$)}  Let $0\leq \ell\leq p-2$. Let  $L_0,\dots,L_p$ be lines passing through a fixed cube $Q$ on scale $\ell+1$ and satisfying $\angle(L_i,L_j)=1$ for each $i\neq j$.  Let $Q'$ be the cube on scale $\ell$ containing $Q$.  Then the set
\[ X=\bigcup_{i=0}^p(L_i\cap Q')\setminus Q \]
is a \emph{fan} on scale $\ell$.
\end{definition}

For dimension $n>2$, we will need a variant of the above configuration involving a $(p+1)$-tuple of lines in a neighbourhood of a 2-plane. We pause for a moment to define the relevant concepts. A \emph{2-plane} in $R^n$ is the linear span over $R$ of any two vectors $u,v\in\mathcal{B}$ such that $\angle(u,v)=1$. 
For a set $S\subset R^n$, and for $j\in[k+1]$, we define the $p^{-j}$-neighbourhood of $S$ by
$$
\mathcal{N}_j(S)=\{x\in R^k:\ \hbox{dist}(x,S)\leq p^{-j}\},
$$
where we say that dist$(x,S)=p^{-\ell}$ if $\ell=\max\{j:\ p^j|(x-s)$ for some $s\in S\}$.

\begin{definition}
\label{def-fan-general}
{\bf (Fans in dimension $n>2$)}
Let $Q'$ be a cube on scale $\ell$ in $R^n$. Define
$$
\Pi:=Q'\cap \mathcal{N}_{\ell+1}(\Pi_0),
$$
where $\Pi_0$ is a 2-plane passing through some point $a\in Q'$. Let $Q=Q_{\ell+1}(a)\subset Q'\cap\Pi$ be the cube on scale $\ell+1$ containing $a$. 
Let $L_0,\dots,L_p\subset R^n$ be lines that pass through $Q$, make pairwise angles $1$, and such that $L_j\cap Q'\subset \Pi$ for each $j$. 
Then 
\[ X=\bigcup_{i=0}^p(L_i\cap Q')\setminus Q \]
is a \emph{fan} on scale $\ell$.

\end{definition}

\begin{theorem}
\label{thm-geotest}
Let $f\in\mathcal{H}^n$ be a hyperplane function, and let $X\subset R^n$ be a fan. Then 
\[ \sum_{x\in R^n}f(x)\one_X(x)=0\bmod{p}.\]
\end{theorem}
To prove the theorem, it suffices to prove that $|H\cap X|=0\bmod{p}$ for any hyperplane $H$ and any fan $X$. We prove this 
in Proposition~\ref{geo-test-for-H}.

\subsection{Preliminary lemmas}

Let $Q$ be a cube on scale $\ell$. For $x\in Q$, write $x=x'+p^\ell x''$ with $x'\in[p^\ell]$ and $x''\in[p^{k-\ell}]$. Note that if  $x,y\in Q$, then (with the obvious notation) we have $y'=x'$. We may therefore identify $Q$ with $R_{k-\ell}^n$ via the map $\iota_Q:Q\rightarrow R_{k-\ell}^n$ defined by
\[ \iota_Q(x'+p^\ell x'')=x''.\]

\begin{lemma}
\label{idQasZp} {\bf (Properties of $\iota_Q$)}  
Let $Q$ be a cube on scale $\ell$ for some $0\leq \ell\leq k-1$. Then:
\begin{enumerate}
\item[(i)] If $L\subset R^n$ is a line in direction $b$ intersecting $Q$, then $\iota_Q(Q\cap L)$ is a line in direction $\pi_{k-\ell}(b)$ in $R_{k-\ell}^n$.

\item[(ii)] If $H\subset R^n$ is a hyperplane with normal direction $b$ intersecting $Q$, then $\iota_Q(Q\cap H)$ is a hyperplane with normal direction $\pi_{k-\ell}(b)$ in $R_{k-\ell}^n$. 

\item[(iii)] If $\Pi\subset R^n$ is a $2$-plane intersecting $Q$, then $\iota_{Q}(\Pi)$ is a $2$-plane in $R_{k-\ell}^n$.

\item[(iv)] If $S\subset Q$, and if $j\geq \ell$, then $\iota_Q(\mathcal{N}_j(S))=\mathcal{N}_{j-\ell}(\iota_Q(S))$.
\end{enumerate}
\end{lemma}

\begin{proof}

Pick some point $a\in Q\cap L$, and suppose $a=a'+p^{\ell}a''$ with $a'\in[p^{\ell}]$. Then 
\[ Q\cap L =\{a+(\lambda p^\ell) b :\lambda\in R_{k-\ell} \}, \]
and so
\[ \iota_Q(Q\cap L)=\{a''+\lambda\pi_{k-\ell}(b):\lambda\in R_{k-\ell}\}\subset R_{k-\ell}^n.\]
Now suppose $H=\{x:\langle x-c,b\rangle=0\}$, and $a\in H\cap Q$. Then 
\[ Q\cap H =\{a+y:y=p^\ell y'', \; \langle y, b\rangle=0\bmod{p^k}\} =\{a+p^\ell y'': \langle y'', b\rangle=0\bmod{p^{k-\ell}}\} \]
and so
\[ \iota_Q(Q\cap H)=\{a''+y'':\langle y'',\pi_{k-\ell}(b)\rangle=0\}\subset R_{k-\ell}^n. \]
The proof of (iii) is similar. Finally, (iv) follows from the observation that for $x,y\in Q$ and for $i\geq\ell$, 
\[ p^i\mid(x-y) \quad \text{if and only if} \quad p^{i-\ell}\mid(\iota_Q(x)-\iota_Q(y)).\]
\end{proof}

\begin{lemma}
\label{lem-pimap}
{\bf (Properties of $\pi_\ell$)} For $0\leq \ell\leq k-1$, the following statements hold:
\begin{enumerate}
\item[(i)] If $L\subset R^n$ is a line in direction $b$, then $\pi_\ell(L)\subset R_\ell^n$ is a line in direction $\pi_\ell(b)$. In particular, if $\angle(L,L')=1$ in $R^n$, then $\angle(\pi_\ell(L), \pi_\ell(L'))=1$ in $R_\ell^n$.
\item[(ii)] If $H\subset R^n$ is a hyperplane with normal direction $b$, then $\pi_\ell(H)\subset R_\ell^n$ is a hyperplane with normal direction $\pi_\ell(b)$.
\item[(iii)] If $\Pi\subset R^n$ is a 2-plane spanned by $b,b'\in\mathcal{B}$, then $\pi_\ell(\Pi)\subset R_\ell^n$ is a 2-plane spanned by $\pi_\ell(b),\pi_\ell(b')$.
\item[(iv)] If $ S\subset R^n$, and if $j\geq \ell$, then $\pi_\ell(\mathcal{N}_j(S))= \pi_\ell(S)$.

\end{enumerate}
\end{lemma}

\begin{proof}

By linearity, if $L=\{a+tb:t\in R\}$ is a line, then
\[ \pi_\ell(L)=\{\pi_\ell(a)+t'\pi_\ell(b):t'\in R_\ell\}. \]
This proves (i). For (ii), suppose $H=\{x:\langle x-a,b\rangle=0\}$. We claim that 
\begin{equation}\label{rescaled-H}
\pi_\ell(H)=\{x'\in R_\ell:\langle x'-a',b'\rangle=0\}.
\end{equation}
Indeed, writing $x=x'+x''p^\ell$, and similarly for $a$ and $b$, we have
\[ \langle x-a,b\rangle = \langle x'-a',b'\rangle +p^\ell(\langle x'-a',b''\rangle\rangle+\langle x''-a'',b\rangle). \]
Applying $\pi_\ell$ to both sides of this equation, and noting that $\pi_\ell(0)=0$, we conclude that $\pi_\ell(H)\subset\{x'\in R_\ell:\langle x'-a',b'\rangle=0\}$. Conversely, suppose $x'\in[p^\ell]$ satisfies $\langle x'-a',b'\rangle=0$ mod $p^\ell$. Then for any $x''$ satisfying 
\[ \langle x'-a',b''\rangle+\langle x''-a'',b \rangle=0\bmod{p^{k-\ell}}, \]
we have $x=x'+x''p^\ell\in H$ (notice that such an $x''$ must exist as $b$ is non-zero mod $p$). This gives (\ref{rescaled-H}) The proof of (iii) is similar. Finally, (iv) follows directly from the definitions of the $p^{-j}$ neighbourhood of a set and the map $\pi_\ell$.
\end{proof}

\begin{lemma}
\label{gen-intersectinQ}
Let $L,L'\subset R^n$ be lines. Assume that $\angle (L, L')=1$, and that $L$ and $L'$ both intersect a cube $Q$ on scale $1$. Then $L\cap L'\subset Q$. 
\end{lemma}
\begin{proof}
Suppose $L$ and $L'$ intersect in some cube $Q'$ on scale $1$. Then the lines $\pi_1(L)$ and $\pi_1(L')$ in $R_1^n$ pass through both of the points $q'=\pi(Q')$ and $q=\pi(Q)$. But Lemma~\ref{lem-pimap} implies that $\pi_1(L)$ and $\pi_1(L')$ make angle $1$, hence intersect uniquely. Therefore $q=q'$, which means that $Q=Q'$, and indeed any intersection points of $L$ and $L'$ lie in $Q$.  
\end{proof}

\begin{lemma}
\label{LcapH}
Let $L\subset R^n$ be a line in direction $b$, and let $H\subset R^n$ be a hyperplane with normal direction $v$. Assume that they intersect, and that $\langle b,v\rangle=cp^j$ for some invertible $c\in R^\times$ and $j\geq 0$. If $j=0$, then the intersection point is unique. If $j>0$, then there is some cube $Q$ on scale $k-j$ so that $L\cap H\subset Q$, and $|L\cap H|=p^j$. 
\end{lemma}

\begin{proof}
Let $a\in L\cap H$, so that $L=\{a+tb:t\in R^n\}$ and $H=\{x:\langle x-a,v\rangle=0\}$. Then 
$L\cap H$ consists of points $x=a+tb$ with $t\in R$ such that 
\[ 0= \langle x-a,v\rangle= t\langle b,v\rangle= tcp^j  \bmod{p^k}. \]
If $j=0$, then we have a unique intersection point with $t=0$. If $j\geq 1$, the intersection points correspond to $t=0 \bmod{p^{k-j}}$, that is, $t=\ell p^{k-j}$ for $\ell \in[p^j]$. This yields $p^j$ intersection points, all in the same cube on scale $k-j$. 
\end{proof}

\subsection{A simplified geometric test}
\label{subsec-geo-test-2}
In this subsection we prove Theorem \ref{thm-geotest} in the simple case when $n=2$ and $\ell=0$; the general case is deferred until the next subsection.

Let $\mathcal{L}=\{L_0,L_1,\dots,L_p\}$ be a collection of $p+1$ lines in $R^2$. Assume that  
there is some square $Q$ on scale $1$ such that $Q\cap L_i\neq\emptyset$ for all $i$, and that 
$\angle(L_i,L_j)=1$ for any $i\neq j$. Notice that, if $B$ is the set of directions of the lines $L_i$, then
\begin{equation}
\label{normsmodp}
\{b \bmod{p}: b\in B\} = \{(0,1),(1,0),(1,1)\dots,(1,p-1)\}. 
\end{equation}
If $L$ is any line in $R^2$, then there is a unique line $L_i$ in $\mathcal{L}$ such that $\angle(L,L_i)<1$. 

\begin{lemma}
\label{LcapH2}
Let $L,L'$ be lines in $R^2$. Let $b$ be the direction of $L$, and $v$ the normal direction of $L'$. 

(i) If $\angle(L,L')=1$, then $\langle v,b\rangle\neq0\bmod{p}$. Consequently, by Lemma~\ref{LcapH}, $L$ and $L'$ have a unique intersection point.

(ii) If $\angle(L,L')<1$, then $\langle v,b\rangle=0\bmod{p}$. Consequently, by Lemma~\ref{LcapH}, for any square $Q$ on scale $1$ we have $|L\cap L'\cap Q|=0\bmod{p}$. 
\end{lemma}

\begin{proof}
For any $b,v\in\mathcal{B}$, the directions $b$ and $v$ mod $p$ must belong to the set on the right side of ~(\ref{normsmodp}). The lemma is now easy to verify directly.
\end{proof}

\begin{proposition}
\label{HPtest2}
Let $\mathcal{L}$ and $Q$ be as described above. Let 
\[ X=\bigcup_{i=0}^pL_i\setminus Q. \]
Then for any line $L$, we have $|L\cap X|=0\bmod{p}$.
\end{proposition}
\begin{proof}
Let $L$ be a line. By the observation in~(\ref{normsmodp}), there is a unique line $L'$ in $\mathcal{L}$ such that $\angle(L,L')<1$.  Without loss of generality, assume $L'=L_0$. 

Notice that for $i\neq j$, the intersection of $L_i$ and $L_j$ is contained in $Q$, by Lemma~\ref{gen-intersectinQ}. Therefore no two distinct lines in $\mathcal{L}$ may intersect in $X$, so that 
\begin{equation} 
\label{splitsum}
|L\cap X|= \sum_{i=0}^p|L\cap L_i\cap X|.
\end{equation}
First suppose that $L\cap Q=\emptyset$. Then by Lemma~\ref{LcapH2}, for each $i\in\{1,\dots,p\}$, $L$ intersects $L_i$ at a unique point $p_i\not\in Q$, so $|L\cap L_i\cap X|=1$ for $i=1,\dots,p$. 
Next we count the size of $L\cap L_0\cap X$. By Lemma~\ref{LcapH2}, the size of $L\cap L_0$ in any square of scale $1$ is $0$ modulo $p$, so $|L\cap L_0\cap X|=0\bmod{p}$.  Combining this all with~(\ref{splitsum}), we obtain $|L\cap X|=0\bmod{p}$, as desired. 

Now suppose $L\cap Q\neq\emptyset$. Then by Lemma~\ref{gen-intersectinQ}, for $i=1,\dots,p$, we have that $L\cap L_i\subset Q$, and so $L\cap L_i\cap X=\emptyset$. Therefore $X\cap L = X\cap L\cap L_0$, and by the same argument as in the previous case, the size of this set is 0 modulo $p$. 
\end{proof}

\noindent\emph{Example.}  We can use the previous proposition to show that in $(\ZZ/4\ZZ)^2$, the phi function $\phi_{21}$ does not lie in the span of hyperplane functions. 
We record the values of $\phi_{21}(x,y)$ in the following table, with rows indexed by $y\in\ZZ/4\ZZ$ and columns by $x\in \ZZ/4\ZZ$:
\begin{displaymath}
\begin{tabular}{c|cc:cc}
\text{} & 0 & 2 & 1 & 3 \\
\hline 
0 & 0 & 0 & 0 & 0 \\
2 & 0 & 0 & 0 & 0 \\
\hdashline
1 & 0 & 1 & 0 & 1 \\
3 & 0 & 1 & 0 & 1 \\
\end{tabular}
\end{displaymath}
We used dashed lines in the table to partition $(\ZZ/4\ZZ)^2$ according to its four squares on scale $1$. Let $Y=\hbox{supp\,}\phi_{21}$. Observe that $\pi_1(Y)=\{(0,1),(1,1)\}\subset(\ZZ/2\ZZ)^2$ is a line in the direction $(1,0)$, whereas for each square $Q$ on scale 1, the set $\iota_Q(Y\cap Q)$ is either empty or else a line in the direction $(0,1)$. In this sense, $Y$ is a line both globally on the rough scale and locally on each square on scale 1, but the directions on the two scales are inconsistent with each other. 

One could ask if there might be a way to represent $\phi_{21}$ as a linear combination of several hyperplane functions. Our geometric test shows that this is in fact impossible.
Take $Q$ to be the square containing the point $(0,1)$. Let $L_0$, $L_1$, $L_2$ be lines in directions $(1,0)$, $(1,1)$, and $(0,1)$, respectively, all passing through the point $(0,1)$.
Let $X=(L_0\cup L_1 \cup L_2) \setminus Q$. Then $X\cap Y=\{(3,1)\}$, and so 
\[ \sum_{(x,y)\in X}\phi_{21}(x,y)=1\neq0\bmod{2}.\]


\subsection{Generalizing the geometric test} 




Let $n>2$ and let $\Pi\subset R^n$ be a $p^{-1}$-neighbourhood of a 2-plane in $R^n$.
For our hyperplane test in $R^n$, we will consider the intersection of a hyperplane $H$ with $\Pi$, and then apply an adapted form of the $2$-dimensional hyperplane test in $\Pi$. The details in the case $\ell=0$ are given in the following proposition. 

\begin{proposition}
\label{geotest-scale1-n}
Let $\Pi\subset R^n$ be as defined above, and let $Q$ be a cube on scale $1$ in $\Pi$.  Let $L_0,L_1,\dots,L_p\subset \Pi$ be lines in $R^n$ all passing through $Q$ and satisfying $\angle(L_i,L_j)=1$ for all $i\neq j$. Let $X=(\bigcup_{j=0}^p L_j)\setminus Q$.  If $H\subset R^n$ is a hyperplane, then $|X\cap H|=0\bmod{p}$.
\end{proposition}

\begin{proof}
Observe that by Lemma~\ref{gen-intersectinQ}, the lines $L_j$ may only intersect in $Q\subset X^c$, and so 
\begin{equation} 
\label{splitX}
|X\cap H|=\sum_{j=0}^p|X\cap L_j\cap H|.
\end{equation}

Let $b$ be the normal direction of $H$ and $b^{(j)}$ the direction of $L_j$. By Lemma \ref{lem-pimap}, $\pi_1(\Pi)$ is a $2$-plane in $R_1^n$ and $\pi_1(H)$ is a hyperplane in $R_1^n$ with normal direction $\pi_1(b)$. Then either $\pi_1(\Pi)\subset\pi_1(H)$, or else $\pi_1(\Pi)\cap\pi(H)$ is a line. In the first case, $\pi_1(L_j)\subset\pi_1(H)$, so that $\langle b^{(j)},b\rangle=0\bmod{p}$ for all $j$. By Lemma~\ref{LcapH}, for each $j$ and in each cube $Q'$ on scale 1, we have $|H\cap L_j\cap Q'|=0\bmod{p}$, which together with~(\ref{splitX}) gives the desired result. 

Thus for the remainder of the proof we assume that $\overline{L}:=\pi_1(H)\cap\pi_1(\Pi)$ is a line, and also that
there is some $j$ so that $|X\cap L_j\cap H|\neq0\bmod{p}$ (as otherwise,~(\ref{splitX}) gives the desired result), in which case Lemma~\ref{LcapH} implies that the size of the intersection is $1$.  For $i\in[p+1]$, let $\overline{L}_i=\pi_1(L_i)$. By Lemma~\ref{lem-pimap}, $\overline{L}_0,\overline{L}_1,\dots,\overline{L}_p$ are lines so that any pair makes angle $1$, and all pass through the point $q=\pi_1(Q)$.
Moreover, each is contained in the $2$-plane $\pi_1(\Pi)$, and so they inherit properties of lines in $R_1^2$.

Thus $\overline{L}$ has the same direction as exactly one of the $\overline{L}_i$, and intersects the other $p$ lines uniquely. Without loss of generality, assume $\overline{L}$ is parallel to $\overline{L}_0$. Since $H$ intersects $L_j$ outside of $Q$, the unique intersection of $\overline{L}_j$ and $\overline{L}$ is not equal to the point $q$, and in particular, $q\not\in\overline{L}$. Thus $\overline{L}_0\cap\overline{L}$ is empty, and so $L_0\cap H$ is empty as well. Moreover, 
\begin{equation}\label{overline-L-intersection}
 |\overline{L}_i\cap\overline{L}|=|(\overline{L}_i\cap\overline{L})\setminus\{q\}| =1 \quad \text{for each $i=1,\dots,p$}, 
 \end{equation}
since $\overline{L}_i$ and $\overline{L}$ intersect uniquely, and the latter line does not intersect $q$.
Also for such $i$, let 
\[ Q_i = \pi_1^{-1}(\overline{L}\cap\overline{L}_i). \]
Then $Q_i$ is a cube on scale $1$ in $\Pi$ that contains $X\cap L_i \cap H=L_i \cap H$. Combining this with~(\ref{splitX}), we have
\[ |X\cap H|=\sum_{i=1}^p|L_i\cap H|=\sum_{i=1}^p|(L_i \cap Q_i) \cap (H\cap Q_i)|.  \]
We will show that $|(L_i \cap Q_i) \cap (H\cap Q_i)|=1$ for $i=1,\dots,p$, which will complete the proof. To this end, choose $i\in\{1,\dots,p\}$, and identify $Q_i$ with $R_{k-1}^n$ via the map $\iota_{Q_i}$. Since this map is a bijection, we prove $|\iota_{Q_i}(L_i)\cap\iota_{Q_i}(H)|=1$. 

By Lemmas~\ref{idQasZp} and \ref{lem-pimap}, if $v$ is the direction of $\iota_{Q_i}(L_i)$, then $\pi_1(v)$ is the direction of $\overline{L}_i$, and if $b$ is the normal direction of $\iota_{Q_i}(H)$, then $\pi_1(b)$ is the normal direction of $\pi_1(H)$. Since $\overline{L}_i$ is contained in $\pi_1(\Pi)$, we have
$$
\overline{L}_i\cap \pi_1(H)= \overline{L}_i\cap \pi_1(\Pi)\cap \pi_1(H)= \overline{L}_i\cap \overline{L}.
$$
By (\ref{overline-L-intersection}), the last intersection is a single point in $R_1^n$. 
It follows by Lemma~\ref{LcapH} that $\langle\pi_1(b),\pi_1(v)\rangle\neq0\bmod{p}$. But then $\langle b,v\rangle\neq0\bmod{p}$, and so the same lemma gives that $\iota_{Q_i}(L_i)$ and $\iota_{Q_i}(H)$ intersect uniquely as well.
\end{proof}

Now we generalize this argument, and the argument for $n=2$, to cubes on other scales. 

\begin{proposition}
\label{geo-test-for-H}
Let $n\geq2$. For any hyperplane $H\subset R^n$ and any fan $X\subset R^n$, we have $|H\cap X|=0\bmod{p}$. 
\end{proposition}
\begin{proof}
First assume $n>2$. Let $X=\bigcup_{i=0}^p(L_i\cap Q')\setminus Q$ be a fan as in Definition~\ref{def-fan-general}.
By Lemma~\ref{idQasZp}, $\iota_{Q'}(L_0),\dots,\iota_{Q'}(L_p)$ are lines 
so that each pair makes a $p$-adic angle $1$, and $\iota_{Q'}(H)$ is a hyperplane. Moreover,  $\iota_{Q'}(L_0),\dots,\iota_{Q'}(L_p)$ are all contained in $\iota_{Q'}(\Pi)$, and each passes through $\iota(Q)$, a cube in $R_{k-\ell}^n$ on scale $1$.  By Lemma~\ref{idQasZp} (iii) and (iv), $\iota_{Q'}(\Pi)$ is the $p^{-1}$-neighbourhood of a $2$-plane in $R_{k-\ell}^n$. 
We may now apply Proposition~\ref{geotest-scale1-n} to conclude $|\iota_{Q'}(H)\cap \iota_{Q'}(X)|=0\bmod{p}$. Since $\iota_{Q'}$ is a bijection, we have $|H\cap X|=0\bmod{p}$. 

The $n=2$ case is similar, although we apply Proposition~\ref{HPtest2} instead.
\end{proof}


\subsection{Parallel lines} \label{sec-parallel lines}

Any hyperplane $H$ in $R^n$ has the following property. Let $Q$ is a cube on some scale $\ell$, and let $L,L'$ be two parallel lines in $R^n$, both passing through $Q$. Then 
$$|L\cap Q\cap  H|\equiv |L'\cap Q\cap  H|\mod  p.
$$
We prove in Proposition \ref{parallel-lines} that a similar property holds for phi functions.

\begin{proposition}\label{parallel-lines}
Let $Q$ be a cube on scale $\ell$ for some $0\leq \ell\leq k-1$. Let $L,L'$ be two lines in $R^n$ in the direction of the same vector $b\in \mathbb{P}R^{n-1}$, and assume that both $L$ and $L'$ pass through $Q$. Then for any function $f\in\Omega_{p^k-1}^n$ we have
$$
\langle \one_{L\cap Q},f\rangle \equiv \langle \one_{L'\cap Q},f\rangle \mod p.
$$

\end{proposition}

\begin{proof}
We prove the proposition under the assumption that $Q=R^n$.
The general case can be deduced from this by rescaling as in the proof of 
Proposition \ref{geo-test-for-H}. The details are left to the interested reader.

We first claim that it suffices to consider the case when $L,L'$ are lines in the direction of $e_1=(1,0,\dots,0)$. Indeed, let
$b\in \mathbb{P}R^{n-1}$ be the common direction vector for $L$ and $L'$. Without loss of generality, we may assume that $b_1\in R^\times$.
Define a linear mapping $U:R^n\to R^n$ by saying that $U(e_1)=b$ and (with the obvious notation) $U(e_j)=e_j$ for $2\leq j\leq n$. In the basis $e_1,\dots,e_n$, $U$ is represented by the matrix
$$
\begin{pmatrix}
b_1& 0 & 0 &\cdots &0 \\
b_2 & 1 &  0 & \cdots & 0 \\
b_3  & 0 & 1 &\cdots & 0 \\
\vdots \\
b_n & 0 & 0 &\cdots & 1
\end{pmatrix}
$$
Since the determinant of this matrix is $b_1\in R^\times$, $U$ is invertible. Moreover, $U^{-1}$ maps lines in the direction of $b$ to lines in the direction of $e_1$. By iterated applications of (\ref{phi-e1}) and Lemma \ref{bxformula}, $f(x)$ and $f(Ux)$ have the same degree. This proves the claim.

It therefore suffices to prove the following: if $L,L'$ are lines in the direction of $e_1$, then for any $\alpha$ with $|\alpha|\leq p^k-1$ we have 
\begin{equation}\label{LL'}
\langle \one_{L},\phi_\alpha \rangle \equiv \langle \one_{L'},\phi_\alpha\rangle \mod p.
\end{equation}
Let $L$ be the line $\{(y,z):\ y\in R\}$ for some $z\in R^{n-1}$. Let also $\alpha=(\beta,\gamma)$ with $\beta\in [p^k]$ and $\gamma\in[p^k]^{n-1}$. Then
\begin{align*}
\langle \one_{L},\phi_\alpha \rangle &= \sum_{y\in R} \phi_\beta(y)\phi_\gamma(z)
\\
&= \phi_\gamma(z) \sum_{y\in R} \prod_{j=0}^{k-1} \phi_{\beta_j}(y_j)
\\
&= \phi_\gamma(z) \prod_{j=0}^{k-1} \left(\sum_{y_j=0}^{p-1}  \phi_{\beta_j}(y_j)\right),
\end{align*}
where $y=\sum y_jp^j$ and $\beta=\sum \beta_jp^j$ are the $p$-adic expansions of $y$ and $\beta$.

If $0\leq \beta_j<p-1$ for some $j$, then by (\ref{eq:phirec}),
\begin{equation}\label{LL'2}
\sum_{y_j=0}^{p-1}  \phi_{\beta_j}(y_j)
= \sum_{y_j=0}^{p-1}  \left( \phi_{\beta_j+1}(y_j+1) -\phi_{\beta_j+1}(y_j)\right)
=0.
\end{equation}

If both of the expressions $\langle \one_{L},\phi_\alpha \rangle $ and $ \langle \one_{L'},\phi_\alpha\rangle$ are zero mod $p$, then (\ref{LL'}) is clearly true. 
On the other hand, if either expression is nonzero mod $p$, it follows from (\ref{LL'2}) that $\beta_j=p-1$ for all $j$. But then $\beta=p^k-1$. Since $\beta+|\gamma|=|\alpha|\leq p^k-1$, it follows that $|\gamma|=0$, so that $\phi_\gamma(z)=1$. But then $\phi_\alpha$ is the characteristic function of the hyperplane $x_1=p^k-1$, and (\ref{LL'}) is again true with both sides equal to 1. This proves the proposition.

 \end{proof}


\section{Acknowledgement}

The first author was supported by NSERC Discovery Grant 22R80520.
The second author was supported by NSERC Discovery Grants 22R80520 and GR010263.

\bibliographystyle{amsplain}

\begin{thebibliography}{[99]}

\bibitem{Ar} B. Arsovski, {\it The $p$-adic Kakeya conjecture}, J. Amer. Math. Soc. 37 (2024), 69--80.


\bibitem{CS1} J. Blasiak, T. Church, H. Cohn, J. A. Grochow, E. Naslund, W. F. Sawin, and C. Umans, {\it On cap sets and the group-theoretic approach to matrix multiplication,} Discrete Analysis 2017:3, 27 pp.



\bibitem{Dhar} M. Dhar, {\it The Kakeya set conjecture over $\ZZ/N\ZZ$ for general $N$}, arXiv:2110.14889, to appear in Advances in Combinatorics.

\bibitem{Dhar2} M. Dhar, {\it Maximal and $(m,\epsilon)$-Kakeya bounds over $\ZZ/N\ZZ$ for general $N$}, preprint, 2022, arXiv:2209.11443 

\bibitem{Dhar3} M. Dhar, {\it $(n,k)$-Besicovitch sets do not exist in $\ZZ_p^n$ and $\hat{\ZZ}^n$ for $k\geq 2$}, preprint, 2023,
arXiv:2312.02495

\bibitem{DD}
M. Dhar, Z. Dvir, {\it Proof of the Kakeya set conjecture over rings of integers modulo square-free $N$}, Combinatorial Theory 1(2021), \# 4, 21pp.

\bibitem{Dvir} Z. Dvir, {\it On the size of Kakeya sets in finite fields}, J. Amer. Math. Soc. 22 (2009), 1093-1097.

\bibitem{DGY11} Z. Dvir, P. Gopalan, S. Yekhanin, {\it Matching Vector codes}, SIAM
Journal on Computing, 40(4) (2011), 1154--1178.


\bibitem{DH13b} Z. Dvir, G. Hu, {\it Matching-Vector families and LDCs over large modulo}.
In {\it Approximation, Randomization, and Combinatorial Optimization. Algorithms and
Techniques}, pages 513–526. Springer, 2013.

\bibitem{EG} 
J. S. Ellenberg, D. Gijswijt, {\it On large subsets of $\FF_q^n$ with no three-term arithmetic progression}, {\it Annals of Math}, 185(1) (2017), 339-343.

\bibitem{GD}J. M. Goethals, P. Delsarte, {\it On a class of majority-logic decodable cyclic codes}, IEEE
Transactions on Information Theory, 14(2) (1968), 182--188.

\bibitem{HW} J. Hickman, J. Wright, {\it The Fourier restriction and Kakeya problems over rings of integers modulo $N$}, Discrete Analysis 2018:11, 54 pp.

\bibitem{Leibman} A. Leibman, {\it Polynomial mappings of groups.} Israel J. Math. 129:29–60 (2002). 

\bibitem{MM} F. J. MacWilliams, H. B. Mann, {\it On the \(p\)-rank of the design matrix of a difference
set,} Inf. Control., 12:474–488, 1968.

\bibitem{CS2} F. Petrov. {\it Combinatorial results implied by many zero divisors in a group ring,} preprint, 2016, arXiv:1606.03256.

\bibitem{Smith} K.J.C. Smith, {\it On the $p$-Rank of the Incidence Matrix of Points and Hyperplanes in a Finite Projective Geometry}, J. Comb. Theory 7 (1969), 122-129.

\bibitem{CS3}
D. Speyer, {\it Bounds for sum free sets in prime power cyclic groups — three ways}, blog post at https://sbseminar.wordpress.com/2016/07/08/bounds-for-sum-free-sets-in-prime-power-cyclic-groups-three-ways/


\bibitem{YGK12} C. Yuan, Q. Guo, H.  Kan, {\it A novel elementary construction of matching
vectors}, Information Processing Letters, 112(12) (2012), 494--496.














\end{thebibliography}

\vfill

\noindent{\sc Department of Mathematics, UBC, Vancouver,
B.C. V6T 1Z2, Canada}

\smallskip

\noindent{\it  ilaba@math.ubc.ca, ctrainor@math.ubc.ca}


\end{document}